\documentclass{article}

\usepackage[preprint]{neurips_2026}

\usepackage[utf8]{inputenc} % allow utf-8 input
\usepackage[T1]{fontenc}    % use 8-bit T1 fonts
\usepackage{hyperref}       % hyperlinks
\usepackage{url}            % simple URL typesetting
\usepackage{booktabs}       % professional-quality tables
\usepackage{amsfonts}       % blackboard math symbols
\usepackage{nicefrac}       % compact symbols for 1/2, etc.
\usepackage{microtype}      % microtypography
\usepackage{xcolor}         % colors

\title{The sublevel Flood bifiltration: towards scalable 2-parameter persistent homology}

\author{%
  Mattéo Clémot\\
  Université Lyon 1\\
  \texttt{matteo.clemot@univ-lyon1.fr} \\
  \And
  Julie Digne \\
  CNRS, Université Lyon 1\\
  \texttt{julie.digne@cnrs.fr} \\
  \And
  Julien Tierny\\
  CNRS, Sorbonne Université \\
  \texttt{julien.tierny@cnrs.fr} \\
}

\usepackage{amsmath,amssymb,amsthm,stmaryrd}
\usepackage[percent]{overpic}
\usepackage{wrapfig}
\usepackage{subcaption}
\usepackage[nameinlink,capitalise]{cleveref}
\usepackage{multirow}

\graphicspath{{figures/}}

\definecolor{customRed}{HTML}{cf244b}
\definecolor{customBlue}{HTML}{493df5}

\newcommand{\ddim}{d}
\newcommand{\hdim}{k}

\newcommand{\bbi}{\mathbb{I}}
\newcommand{\bbk}{\mathbb{K}}
\newcommand{\bbn}{\mathbb{N}}
\newcommand{\bbr}{\mathbb{R}}

\newcommand{\cali}{\mathcal{I}}
\newcommand{\calk}{\mathcal{K}}

\newcommand{\calw}{\mathcal{W}}

\DeclareMathOperator{\conv}{conv}

\DeclareMathOperator{\nrv} {nrv}

\DeclareMathOperator{\vor} {Vor}

\newcommand{\del}{\ensuremath{\mathrm{Del}}}

\newcommand{\cech}{\v{C}ech}

\newcommand{\deln}{\mathcal{D}}

\newcommand{\offsetf}[1]{\mathcal{O}_{#1}}
\newcommand{\cechf}[1]{\mathcal{\check{C}}_{#1}}

\newcommand{\alphaf}[1]{\mathcal{D}_{#1}}
\newcommand{\ripsf}[1]{\mathcal{R}_{#1}}

\newcommand{\homol}[1]{H_{#1}}
\newcommand{\dgm}[1]{\mathrm{dgm}_{#1}}

\DeclareMathOperator{\bottle}{\mathrm{d_b}}

\DeclareMathOperator{\hausdo}{\mathrm{d_{H}}}
\DeclareMathOperator{\interl}{\mathrm{d_{I}}}
\newcommand{\wasser}[1]      {\calw_{#1}}

\DeclareMathOperator{\flood}{\mathcal{F}}
\DeclareMathOperator{\cflood}{\widetilde{\mathcal{F}}}
\DeclareMathOperator{\floodmodule}{\mathbb{F}}

\DeclareMathOperator{\pareto}{Pareto}

\newcommand{\ball}{B}

\newcommand{\incr}{\cali}
\DeclareMathOperator{\matchd}{\mathrm{d_{match}}}

\DeclareMathOperator{\fflood}{fl}
\newcommand{\dist}{\mathrm{d}}

\newcommand{\vstar}{\mathrm{V}}
\newcommand{\cell}{\mathrm{C}}
\newcommand{\oom}{OOM}

\newtheorem{theorem}{Theorem}

\newcommand{\np}{\textsf{NP}}

\DeclareMathOperator{\supp}{\mathrm{supp}}
\DeclareMathOperator{\bigrades}{\phi}

\newcommand{\lipconst}{\ell}

\newcommand{\changed}[1]{#1}

\newcommand{\Xincr}[1]{X_{\leq{#1}}}

\begin{document}

\maketitle

\begin{abstract}
  Multiparameter persistent homology is a rapidly developing branch of topological data analysis that improves the robustness of single-parameter persistent homology to outliers, while still capturing the metric characteristics of the data.
  However, a notable limitation is its lack of scalability.
  In this paper, we introduce a novel approach for efficiently computing 2-parameter persistent homology on large point sets.
  Our work extends the Flood filtration, originally developed for single-parameter persistence.
  Our construction, called the \textit{sublevel Flood bifiltration}, offers a scalable approximation of the sublevel offset bifiltration.
  We show that it benefits from theoretical stability properties and describe how to compute it efficiently.
  We demonstrate the performance of our approach in classification tasks on low-dimensional synthetic datasets, where density awareness is critical, as well as on real-world time series datasets.
  %Our code will be made publicly available soon.
\end{abstract}

%Multiparameter persistent homology is a rapidly developing branch of topological data analysis that improves the robustness of single-parameter persistent homology to outliers, while still capturing the metric characteristics of the data. However, a notable limitation is its lack of scalability. In this paper, we introduce a novel approach for efficiently computing 2-parameter persistent homology on large point sets. Our work extends the Flood filtration, originally developed for single-parameter persistence. Our construction, called the sublevel Flood bifiltration, offers a scalable approximation of the sublevel offset bifiltration. We show that it benefits from theoretical stability properties and describe how to compute it efficiently. We demonstrate the performance of our approach in classification tasks on low-dimensional synthetic datasets, where density awareness is critical, as well as on real-world time series datasets.

\section{Introduction}

Topological data analysis (TDA) has lately emerged as a toolbox to handle geometric and topological properties of data, using ideas from algebraic topology.
It provides concise and stable representations encoding topological information that is often not captured by classical approaches.
It has been successfully used in machine learning, in a variety of domains and tasks, including
dimensionality reduction \cite{moor2020topological,trofimov2023learning}, classification \cite{reininghaus2015stable} or AI interpretability \cite{gardinazzi2025persistent, fay2026shape} (see \cite{hensel2021survey} for a recent survey).
One of the main tools of TDA is \textit{persistent homology} (PH), which considers \textit{filtrations}, typically nested sequences of simplicial complexes that represent the data at different scales. The evolution of their topological features (connected components, loops, cavities) as the scale parameter increases is then summarized into concise and stable representations, like the \textit{persistence barcode} or the \textit{persistence diagram}.
This single-parameter framework is now mathematically well-understood \cite{edelsbrunner2010computational}.

\paragraph{Single-parameter persistent homology}
For the analysis of Euclidean point sets, common filtrations include the Rips and \cech\ filtrations.
However, as their number of simplices increases rapidly with the number of vertices (exponentially in general, cubically when considering only their 2-skeletons for 1-dimensional PH),
these filtrations become computationally intensive both in memory and runtime, which often limits the size of the input.
To address this limitation, several approaches have been proposed: algorithmic optimization~\cite{bauer_ripser_2021} with efficient, parallel implementation on CPU~\cite{perez2021giotto} or GPU~\cite{zhang2020gpu};
sparse linear-size approximations~\cite{sheehy2012linear};
edge collapses for reducing flag filtrations~\cite{boissonnat2020edge,glisse2022swap};
and for low-dimensional Euclidean point sets, Delaunay-based filtrations
(such as Delaunay, Delaunay--\cech\ \cite{bauer2017morse} and Delaunay--Rips \cite{mishra2023stability, clemot2025delaunay} filtrations) which feature far fewer simplices.
More recently, Graf et al.~introduced the \textit{Flood filtration}~\cite{graf2025floodcomplex}. For an input point set $X$, this filtration is built on the Delaunay complex of a smaller set of landmarks $L$, yet it accounts for the entire set $X$, since each simplex is introduced as soon as its convex hull is fully covered by balls of radius $r$ centered at points in $X$.
It is stable under perturbations of $X$, and, in $\bbr^2$, is proven to approximate
the offset filtration of $X$.
In practice, this approach allows the computation of filtration values to be parallelized on GPUs, resulting in significant acceleration compared to the Delaunay--\cech\ PH, for large point sets.

\begin{figure}
  \centering
  \begin{subfigure}[b]{0.23\textwidth}
      \includegraphics[width=\textwidth]{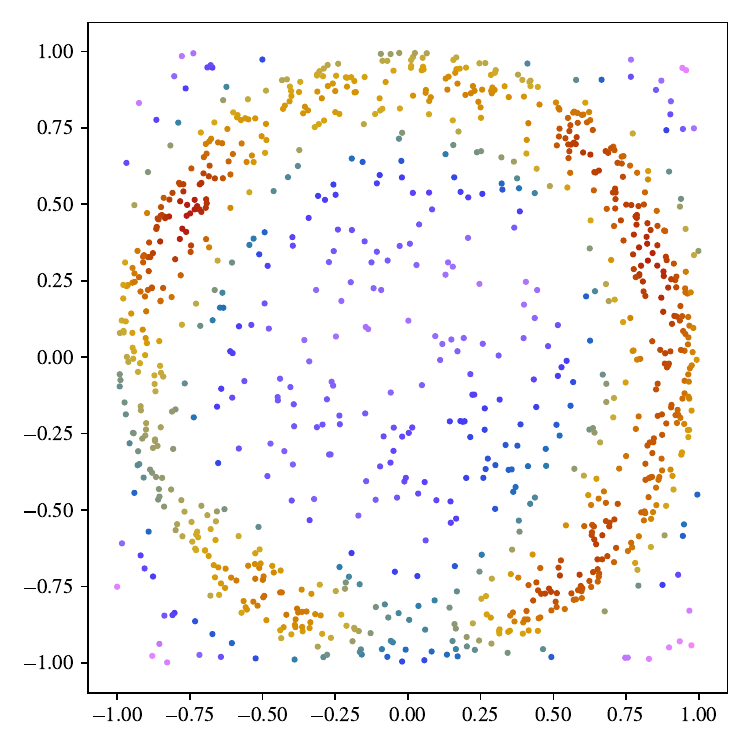}
      \caption{}
  \end{subfigure}
  \hfill
  \begin{subfigure}[b]{0.25\textwidth}
      \includegraphics[width=\textwidth]{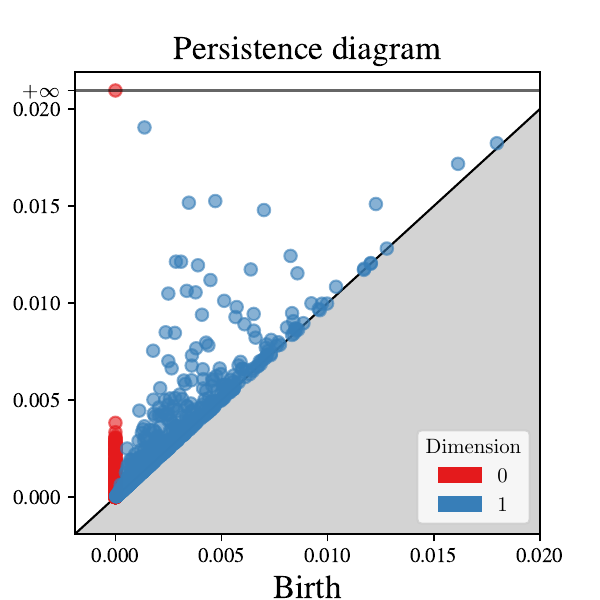}
      \caption{}
  \end{subfigure}
  \begin{subfigure}[b]{0.25\textwidth}
      \includegraphics[width=\textwidth]{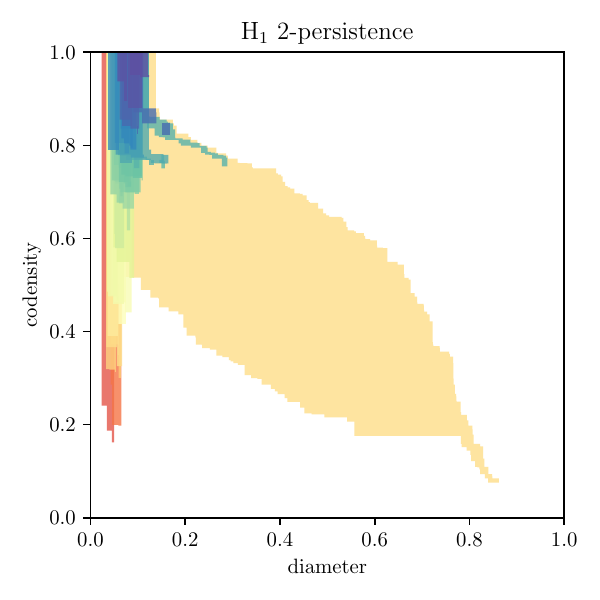}
      \caption{}
  \end{subfigure}
  \hfill
  \begin{subfigure}[b]{0.25\textwidth}
      \includegraphics[width=\textwidth]{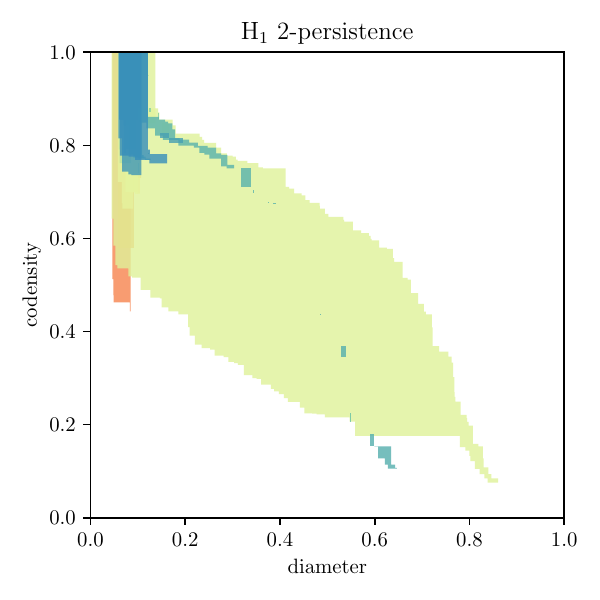}
      \caption{}
  \end{subfigure}
  \caption{
    \textbf{(a)} A point set $X\subset\bbr^2$ colored by its kernel density estimate (KDE).
    \textbf{(b)} The persistence diagram of its Delaunay--(\cech) filtration
    captures many small cycles (blue dots) with similar persistence (between $0.01$ and $0.02$).
    \textbf{(c)} Multiparameter module approximation (MMA) of the sublevel Delaunay bifiltration.
    \textbf{(d)} MMA of our sublevel Flood bifiltration (100 landmarks). The yellow (c) and green (d) large shapes both distinctively capture the prominent ring (a) with high density and large diameter (up to $0.8$).}
  \label{fig_teaser}
\end{figure}

\paragraph{Multiparameter persistent homology}

Single-parameter PH, especially for diameter-based filtrations, is highly sensitive to outliers.
To mitigate this limitation, while still capturing the metric properties of the input data, one can consider more than one parameter, for instance by filtering with respect to both the geometric scale and some density estimate (see \cref{fig_teaser}). This has given rise to multiparameter PH~\cite{carlsson2007theory}, which turns out to be more challenging mathematically and computationally.

\paragraph{Bifiltrations}
One way to define a density-aware bifiltration is to introduce a second (discrete) parameter that encodes some notion of degree or connectivity.
Examples include the degree-Rips and degree-\cech\ bifiltrations~\cite{lesnick2015interactive}, the multicover bifiltration~\cite{edelsbrunner2021multi,corbet2023computing} and its sparse linear-size approximations~\cite{alonso2025sparse}, and the Core and Delaunay--Core bifiltrations~\cite{blaser2025core}. These constructions are all interleaved and benefit from stability properties with respect to the Prohorov distance~\cite{blumberg2024stability}.
Another approach is to build a bifiltration over a point set $X$ using an explicit function $\gamma:X\to\bbr$ accounting for the point set density. $\gamma$ is typically an estimate of codensity (i.e., the opposite of a density estimation, for instance obtained by kernel density estimation (KDE) or by distance-to-a-measure (DTM)~\cite{chazal2011geometric,chazal2018robust}) but can also be an estimation of eccentricity, of curvature, or even an application-provided function (e.g., the partial charge of atoms in a molecule~\cite{keller2018persistent}).
This offers more flexibility, but at the cost of no longer being parameter-free.
In this setting, the theoretical reference bifiltration is the \textit{sublevel offset bifiltration}, which, at scale $(r,s)$, consists of the union of balls of radius $r$ centered at points $x\in X$ such that $\gamma(x)\leq s$, i.e.,
$\offsetf{r,s}(X,\gamma)=\bigcup_{x\in X_s}B_r(x)$ where  $X_s=\gamma^{-1}((-\infty,s])$ (see \cref{fig:sublevel_flood_illustration}).
In practice, one can consider the \textit{sublevel Rips bifiltration}~\cite{botnan2023introduction} given by $\ripsf{r,s}(X,\gamma)=\ripsf{r}(X_s)$.
In low-dimensional Euclidean spaces, the \textit{sublevel Delaunay bifiltration}~\cite{alonso2024delaunay} is far smaller and far more efficient. It constructs a bifiltration on the \textit{incremental Delaunay complex} $\incr(\gamma)$, which solves the non-monotonicity issue of the Delaunay complex, and is topologically equivalent to the sublevel offset bifiltration.

\paragraph{Representations}
Applying the homology functor on a bifiltration yields its associated multiparameter persistence module.
Unlike the single-parameter case, the multiparameter framework lacks a complete, discrete invariant analogous to the persistence barcode~\cite{carlsson2007theory}. This is because multiparameter modules cannot generally be decomposed into a direct sum of interval modules, except in some very specific cases~\cite{botnan2022rectangle,cochoy2020decomposition}.
For instance, the \textit{fibered barcode}~\cite{lesnick2015interactive} of a multiparameter persistence module $M$, which consists of the collection of (single-parameter) persistence barcodes of 1-dimensional affine slices of $M$, does not fully describe $M$.
Several \textit{candidate decomposition} methods \cite{asashiba2023approximation,loiseaux2025multi} have been proposed to alleviate this by approximating $M$ as a direct sum of interval modules.
To incorporate multiparameter PH indicators into machine learning pipelines, various vectorizations have been proposed. These include the multiparameter persistence kernel (MPK)~\cite{corbet2019kernel} and landscapes (MPL)~\cite{vipond2020multiparameter} that are both defined on the fibered barcode. Multiparameter persistence images (MPI)~\cite{carriere2020multiparameter} leverage vineyards to enhance discrimination between modules that share the same fibered barcode. The stable candidate decomposition representation (SCDR)~\cite{loiseaux2023framework} relies on candidate decomposition and comes with stability guarantees.
Finally, graphcodes~\cite{kerber2024graphcode} enable vectorization-free learning from multiparameter persistent homology.

\paragraph{Contributions}
Current approaches for computing 2-parameter persistent homology prove too computationally intensive for practical use. In this work, we tackle this challenge by introducing the \textit{sublevel Flood bifiltration}, along with its theoretical foundations and practical algorithm.
This construction is significantly smaller than existing bifiltrations and leads to orders-of-magnitude speedups for downstream operations (e.g., module candidate decomposition and vectorization).
More precisely, our contributions are as follows:
\begin{itemize}
  \item we introduce the \textit{sublevel Flood bifiltration}, establish its stability properties, and empirically assess its approximation of the sublevel offset bifiltration;
  \item we propose an efficient, GPU-friendly computational approach for this bifiltration (our source code is available at \url{https://github.com/MClemot/SLFlood} and a Python package is available on PyPI at \url{https://pypi.org/project/sublevel-flood/});
  \item we demonstrate, through classification experiments on both synthetic and real-world time series datasets, that our method matches the performance of existing sublevel-based bifiltrations while being far more compact.
\end{itemize}

\section{Background}

We consider a point set $X$ in $\bbr^d$, with $d$ small enough (typically, $d\leq10$) to allow for the computation of the Delaunay complex.

We quickly recall below the foundations of (single-parameter) persistent homology and refer the reader to \cite{edelsbrunner2010computational, chazal2021introduction} for further details.
A \textit{filtration} is a nested sequence of simplicial complexes $\calk=(\calk_r)_{r\geq0}$% that is increasing with respect to the inclusion
, i.e., such that for any $r\leq r'$, $\calk_r\subseteq\calk_{r'}$.
The $\hdim$-th homology group $\homol{\hdim}(\calk_r)$ is a vector space representing the $\hdim$-dimensional topological features of $\calk_r$ (i.e., connected components for $\hdim=0$, loops for $\hdim=1$, cavities for $\hdim=2$).
The core idea of persistent homology is to apply the homology functor to this sequence of simplicial complexes: as $r$ increases, the topological features emerge and vanish.
This information is summarized as the multiset $\dgm{\hdim}(\calk)=\{(b_i,d_i)\mid i\in\llbracket1,m\rrbracket\}$ where $b_i$ and $d_i$ are the filtration values at which a $\hdim$-dimensional topological feature is born and dies, respectively.
Algebraically, this corresponds to the persistence module $\bbk_\hdim=(\homol{\hdim}(\calk_r))_{r\geq0}$ (with linear maps induced by simplicial inclusions) admitting a decomposition into a direct sum of interval modules:
$\bbk_\hdim=\bigoplus_{i=1}^m\bbi[b_i,d_i]$~\cite{chazal2016structure}.

A notable filtration is the \textit{offset filtration} $\offsetf{r}(X) = \bigcup_{x\in X}B_r(x)$. Due to its non-combinatorial nature, one usually rather considers its nerve, known as the \cech\ filtration $\cechf{r}(X) = \left\{\sigma\subset X\mid\bigcap_{x\in\sigma}B_r(x)\neq\varnothing\right\}=\nrv\left(\left(B_r(x)\right)_{x\in X}\right)$. In low-dimensional Euclidean spaces, it is computationally more efficient to use the alpha (or Delaunay) filtration defined as $\alphaf{r}(X) = \nrv\left(\left(B_r(x)\cap\vor(x,X)\right)_{x\in X}\right)$ which is homotopy equivalent to $\offsetf{r}(X)$.

\subsection{The Flood filtration}

The Flood complex was introduced in~\cite{graf2025floodcomplex} to improve the scalability of PH when the size of $X$ increases. It relies on a set of landmarks $L$ that is often---but not necessarily---a subset of $X$, e.g., via furthest point sampling. At scale $r$, it contains all simplices in the Delaunay complex $\del(L)$ that are covered by the union of balls of radius $r$, centered at the points of $X$:
\begin{equation}
  \label{eq:flood_complex}
  \flood_r(X,L)=\left\{\sigma\in\del(L)\mid\conv(\sigma)\subset\bigcup\limits_{x\in X}\ball_r(x)\right\}
\end{equation}
where $\conv(\sigma)$ is the convex hull of $\sigma$.
The persistence diagrams of the Flood filtration are bottleneck-stable with respect to $X$, i.e., for any $X,Y,L$ and any $k\in\bbn$:
\begin{equation}
  \label{eq:flood_stability}
  \bottle\left(\dgm{\hdim}(\flood(X,L)), \dgm{\hdim}(\flood(Y,L))\right)\leq\hausdo(X,Y)
\end{equation}
where $\bottle$ is the bottleneck distance between persistence diagrams, and $\hausdo$ is the Hausdorff distance.
Besides, in the specific case where $L,X\subset\bbr^2$, it approximates the topology of the offset filtration:% meaning that for any $\hdim\in\bbn$:
\begin{equation}
  \label{eq:flood_approx}
  \forall\hdim\in\bbn, \bottle\left(\dgm{\hdim}(\flood(X,L)), \dgm{\hdim}(\alphaf{}(X))\right)\leq2\hausdo(X,L).
\end{equation}
This guarantees that the approximation of the topology of $\alphaf{}(X)$ by the Flood filtration is at most twice the error induced by a subsampling strategy (i.e., computing the PH of $\alphaf{}(L)$). In practice, however, the approximation turns out to be significantly better than this theoretical bound. This improvement arises because the filtration accounts for the entire $X$ (through the consideration of the union of balls centered at $X$), despite being defined only on $\del(L)$.

\subsection{2-parameter persistent homology}

\paragraph{Bifiltrations}
A \textit{bifiltration} is a simplicial complex $\calk$ and a collection of subcomplexes $(\calk_{\alpha})_{\alpha\in\bbr^2}$ that is increasing in the sense that $\calk_{\alpha}\subseteq\calk_{\beta}$ whenever $\alpha\preceq\beta$ (which means $\alpha_1\leq\beta_1$ and $\alpha_2\leq\beta_2$). The \textit{support} of a simplex $\sigma\in\calk$ is $\supp(\sigma)=\{\alpha\in\bbr^2\mid\sigma\in\calk_\alpha\}$.
This general definition can be adapted to our specific context, where, informally, each simplex is \textit{inserted} into $\calk$ along a \textit{staircase} (i.e., each support is a finite union of upper-right quadrants).
More formally, each simplex has a finite,
\begin{wrapfigure}{r}{0.3\linewidth}
  \vspace{-1\baselineskip}
  \hfill
  \begin{overpic}[width=.97\linewidth]{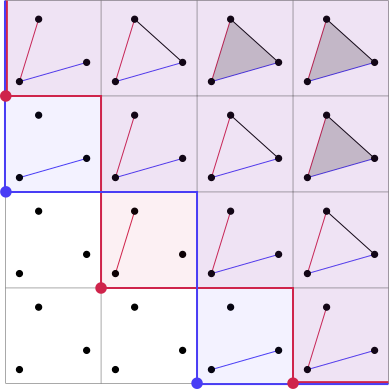}
    \put(-1,-6){\scriptsize0}
    \put(24,-6){\scriptsize1}
    \put(49,-6){\scriptsize2}
    \put(74,-6){\scriptsize3}
    \put(-6,-1){\scriptsize0}
    \put(-6,24){\scriptsize1}
    \put(-6,49){\scriptsize2}
    \put(-6,74){\scriptsize3}
  \end{overpic}
  % \caption{A 3-critical bifiltration.}
\end{wrapfigure}
minimal set of \textit{bigrades}
$\bigrades(\sigma)\subset\bbr^2$ such that $\sigma\in\calk_\alpha$ if and only if there exists $\beta\in\bigrades(\sigma)$ such that $\beta\preceq\alpha$.
This minimal set of bigrades $\phi(\sigma)$ corresponds to the Pareto front%
\footnote{The Pareto front of $S\subset\bbr^2$ is defined as $\pareto(S)=\{\alpha\in S\mid\nexists\beta\in S,\beta\preceq\alpha\text{ and }\alpha\neq\beta\}$.}
of $\supp(\sigma)$
and represents the multiple appearances of $\sigma$ within $\calk$ along the staircase.
This is analogous to the simplex \textit{insertion} in single-parameter persistence.
The bifiltration is said to be \textit{$k$-critical} if every $\sigma\in\calk$ has a minimal set of bigrades of size $|\phi(\sigma)|\leq k$.
For instance, the sublevel Rips and Delaunay bifiltrations are always 1-critical while the degree-Rips bifiltration is multi-critical.
The inlined figure illustrates a 3-critical example, where the support of the red (respectively blue) edge of $\calk$ is shaded in red (resp. blue) with $3$ (resp. $2$) bigrades marked as red (resp. blue) dots at the corners of the support.

\paragraph{Multiparameter persistence modules}

Applying the $\hdim$-th homology functor $\homol{\hdim}$ on a bifiltration yields the associated 2-parameter persistence module $\bbk=\left(\homol{\hdim}(\calk_\alpha)\right)_{\alpha\in\bbr^2}$.
It can be viewed as a collection of vector spaces equipped with linear maps $\homol{\hdim}(\calk_\alpha)\to\homol{\hdim}(\calk_\beta)$ for each $\alpha\preceq\beta$ induced by the simplicial inclusions $\calk_\alpha\subset\calk_\beta$.
The interleaving distance $\interl$ between (isomorphism classes of) 2-parameter
persistence modules generalizes from the single-parameter case~\cite{lesnick2015theory}.
However, its computation is \np-hard~\cite{bjerkevik2020computing}.
A more practical alternative is the matching distance $\matchd$~\cite{cerri2013betti}, a distance between fibered barcodes, providing a lower bound for the interleaving distance \changed{(see \cref{sec_additional_background} for details)}.
%Both exact~\cite{kerber2019exact} and approximate~\cite{kerber2020efficient} algorithms have been proposed to compute it.

\paragraph{Module preprocessing}

A common preprocessing step is to compute a \textit{minimal presentation}~\cite{lesnick2022computing,kerber2021fast}, which reduces the size of the input by providing the smallest possible representation of a multiparameter persistence module. This often accelerates subsequent algorithms, such as candidate decomposition.
However, these minimal presentation algorithms require \textit{free}, i.e., \textit{1-critical}, bifiltrations.
To handle multi-critical bifiltrations, \cite{chacholski2017combinatorial} proposed an algorithm to compute a \textit{free implicit representation}. More recently,~\cite{bauer2025fast} %proposed to work on the chain complex level by 
proposed two \textit{one-criticalization} algorithms that both compute a \textit{free resolution} of the input, i.e., a free chain complex that is isomorphic in homology to the input.

See \cref{sec_additional_background} for further background on multiparameter persistent homology.

\section{The sublevel Flood bifiltration}
\label{sec_theory}

This section presents our novel bifiltration for the efficient computation of 2-parameter persistent homology. Proofs are provided in \cref{sec_proofs}.

\subsection{Definition}

In order to extend the Flood filtration to the 2-parameter framework, we define the \textit{sublevel Flood bifiltration} for $X,L\subset\bbr^d$ and a function $\gamma:X\to\bbr$, writing $X_s=\gamma^{-1}((-\infty,s])$, as follows:
\begin{equation}
  \label{eq:function_flood}
  \flood_{r,s}(X,L,\gamma)=\flood_{r}\left(X_s,L\right)=\left\{\sigma\in\del(L)\mid\conv(\sigma)\subset\bigcup\limits_{x\in X_s}\ball_r(x)\right\}.
\end{equation}
This bifiltration is multi-critical but uses a small complex, in contrast to sublevel Rips and Delaunay bifiltrations which are 1-critical but require a large complex to describe the data.
%This bifiltration is multi-critical: it can be considered opposite to sublevel Rips and Delaunay bifiltrations in the compromise between small complex with high criticality and large complex with 1-criticality to describe the data.
\cref{fig:sublevel_flood_illustration} illustrates this construction on a planar point set.

We suppose in the following that $X=(x_1,\ldots,x_n)$ is sorted in increasing order of $\gamma$.
We also write for convenience $\gamma_i=\gamma(x_i)$ and $\Xincr{i}=X_{\gamma_i}=(x_1,\ldots,x_i)$.
The support of a simplex $\sigma\in\del(L)$ is then described by the following set of bigrades: 
\begin{equation}
  %\bigrades(\sigma)=\left\{\left(\max_{p\in\conv\sigma}d\left(p,X_i\right), \gamma_i\right),1\leq i\leq n\right\}
  \bigrades(\sigma)=\left\{\left(\fflood(\sigma,\Xincr{i}),\gamma_i\right),1\leq i\leq n\right\}
  \text{ with }
  \fflood(\sigma,X)=\max_{p\in\conv\sigma}\dist(p,X)=\max_{p\in\conv\sigma}\min_{x\in X}\dist(p,x)
  \label{eq:true_bigrades}
\end{equation}
where $\dist$ is the Euclidean distance. Taking the Pareto front of this set gives the minimal set of bigrades of $\sigma$. 
In the worst case, its size can reach $|X|$, which means that $\flood(X,L,\gamma)$ is, in general, $|X|$-critical.
In practical cases, however, the criticality turns out to be much smaller.

In the following, we write $\floodmodule(X,L,\gamma)$ for the 2-parameter persistence module obtained by applying the homology functor to the Flood bifiltration: $\floodmodule(X,L,\gamma)=\left(\homol{*}(\flood_{r,s}(X,L,\gamma))\right)_{r,s}$ with the maps induced by the simplicial inclusions.

\begin{figure}
  \centering
  \rotatebox{90}{\hspace{2.7cm}codensity $\longrightarrow$}
  \includegraphics[width=.97\linewidth]{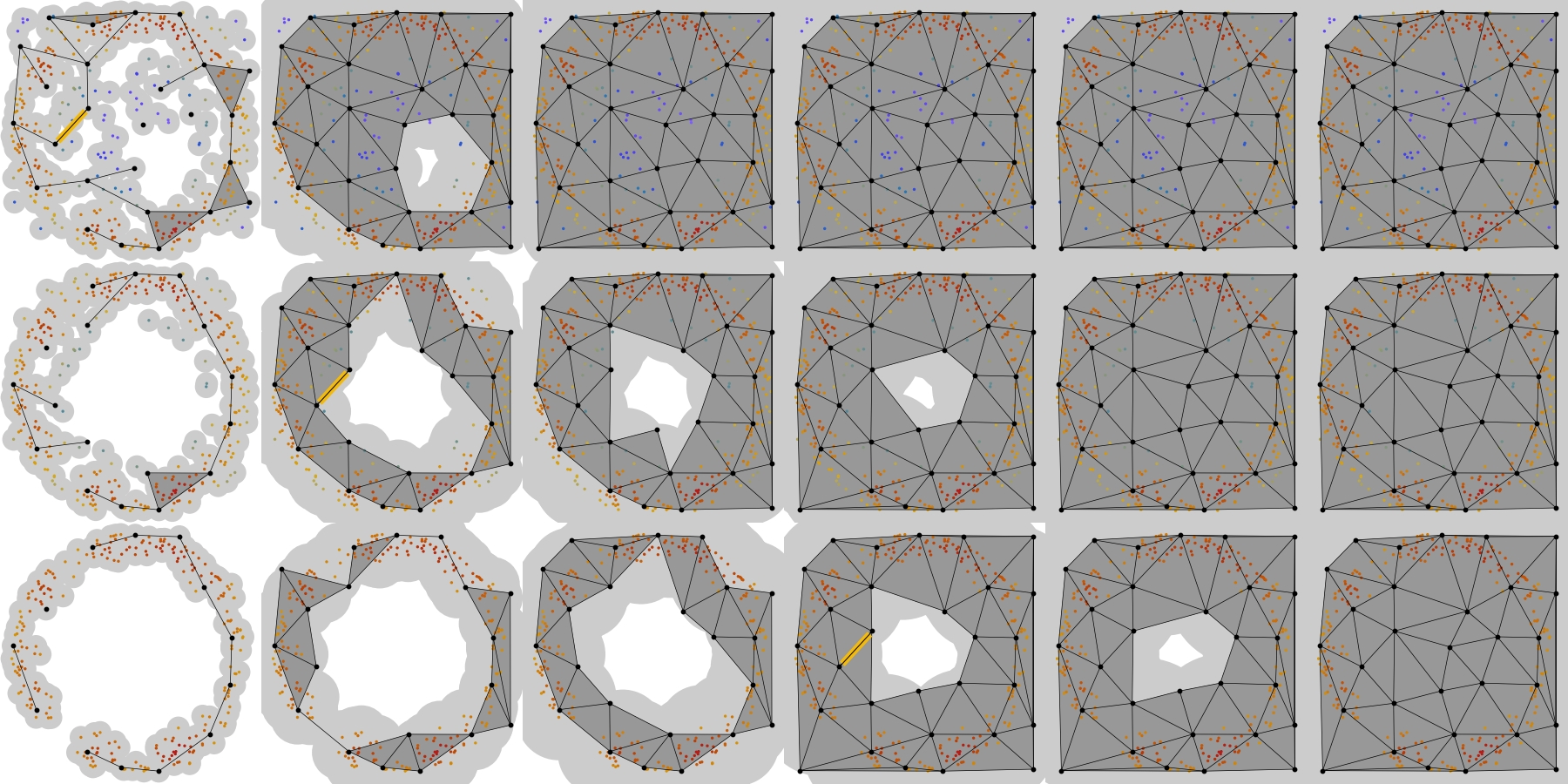}
  radius $\longrightarrow$
  \caption{$\flood_{r,s}(X,L,\gamma)$ is represented with black vertices, black edges and dark gray triangles for $L\subset X\subset\bbr^2$ ($|X|=400$, $|L|=30$), and $\gamma:X\to\bbr$ a codensity function obtained by a Gaussian kernel density estimator.
  The edge highlighted in yellow illustrates the multi-criticality of the construction, as its minimal set of bigrades is of size at least 3.
  The sublevel offset bifiltration $\offsetf{r,s}(X,\gamma)=\offsetf{r}(X_s)$ is shown in light gray.}
  \label{fig:sublevel_flood_illustration}
\end{figure}

\subsection{Properties}
\label{sec:properties}
A desirable property for a topological indicator is stability under input perturbation.
We quantify this stability by providing an upper bound on the interleaving distance between the sublevel Flood persistence module of $X$ and that of its perturbation $Y$, with a fixed set of landmarks.
This, in turn, implies the same bounds for the matching distance between the associated fibered barcodes.

\begin{theorem}[Stability w.r.t. $X$ and $\gamma$]
  \label{th:stability_X_gamma}
  Let $X,Y\subset\bbr^d$ be two point sets. Suppose that both $\gamma_X$ and $\gamma_Y:\bbr^d\to\bbr$ are $\lipconst$-Lipschitz.
  With $L$ fixed, the interleaving distance between both sublevel Flood modules is upper bounded by:
  \begin{equation}
    \interl\left(\floodmodule(X,L,\gamma_X), \floodmodule(Y,L,\gamma_{Y})\right)\leq\max\left(\hausdo(X,Y),\lipconst\hausdo(X,Y)+\lVert\gamma_X-\gamma_Y\rVert_\infty\right).
  \end{equation}
\end{theorem}
% \begin{proof}
%   \remark{will be moved in appendix}
%   Let $\sigma\in\flood_{r,s}(X,L,\gamma_X)$ and $p\in\conv\sigma$. There is a $x\in X$ such that $\dist(p,x)\leq r$ and $\gamma_X(x)\leq s$.
%   Then, there exists $y\in Y$ such that $\dist(x,y)\leq\hausdo(X,Y)$. Therefore $\dist(p,y)\leq\dist(p,x)+\dist(x,y)\leq r+\hausdo(X,Y)$.
%   Besides:
%   \begin{align*}
%     \gamma_Y(y)-s&\leq\gamma_Y(y)-\gamma_X(x)\\
%     &\leq|\gamma_Y(y)-\gamma_X(x)|\\
%     &\leq|\gamma_Y(y)-\gamma_Y(x)|+|\gamma_Y(x)-\gamma_X(x)|\\
%     &\leq\lipconst\hausdo(X,Y)+\lVert\gamma_X-\gamma_Y\rVert_\infty
%   \end{align*}
%   This means that $\sigma\in\flood_{r+\hausdo(X,Y),s+\lipconst\hausdo(X,Y)+\lVert\gamma_X-\gamma_Y\rVert_\infty}(Y,L,\gamma_Y)$. By writing $\alpha=\max\left(\hausdo(X,Y),\lipconst\hausdo(X,Y)+\lVert\gamma_X-\gamma_Y\rVert_\infty\right)$, we have:
%   \begin{align*}
%     \flood_{r,s}(X,L,\gamma_X)&\subseteq\flood_{r+\alpha,s+\alpha}(Y,L,\gamma_Y)\\
%     \flood_{r,s}(Y,L,\gamma_Y)&\subseteq\flood_{r+\alpha,s+\alpha}(X,L,\gamma_X)
%   \end{align*}
%   These simplicial inclusions provide an $\alpha$-interleaving between the two modules $\floodmodule(X,L,\gamma_X)$ and $\floodmodule(Y,L,\gamma_{Y})$, hence the result.
% \end{proof}

As mentioned earlier, $\gamma$ is often a codensity estimation obtained by KDE or DTM.
In this case, we can refine the stability result when each point of $X$ is displaced.
\begin{theorem}[Stability w.r.t. $X$]
  \label{th:stability_X}
  %Let $X\subset\bbr^d$ a point set and its perturbed version $Y=X+E$ where each $E_i\in\bbr^{d}$ is such that $\lVert E_i\rVert_2\leq\varepsilon$. We consider the case where $L$ is fixed and $\gamma_X$ and $\gamma_Y$ are density estimates obtained by either KDE or DTM:
  Let $X\subset\bbr^d$ be a point set and $Y$ its $\varepsilon$-perturbation such that $\lVert x_i-y_i\rVert_2\leq\varepsilon$ for all $i$.
  % We consider the case where $L$ is fixed and $\gamma_X$ and $\gamma_Y$ are density estimates obtained by either KDE or DTM:
  % \begin{itemize}
  %   \item when $\gamma$ is obtained by KDE with a $\lipconst_K$-Lipschitz kernel $K$ and bandwidth $h$:
  %   \[\interl\left(\floodmodule(X,L,\gamma_X), \floodmodule(Y,L,\gamma_{Y})\right)\leq\varepsilon\max\left(1,\frac{2\ell_K}{h^{d+1}}\right);\]
  %   \item when $\gamma$ is obtained by DTM with parameter $m\in(0,1]$:
  %   \[\interl\left(\floodmodule(X,L,\gamma_X), \floodmodule(Y,L,\gamma_{Y})\right)\leq\varepsilon\left(1+\frac{1}{\sqrt{m}}\right).\] 
  % \end{itemize}
  We consider the case where $L$ is fixed.
  If $\gamma_X$ and $\gamma_Y$ are obtained by a KDE with an $\lipconst_K$-Lipschitz kernel $K$ and bandwidth $h$, then:
    \[\interl\left(\floodmodule(X,L,\gamma_X), \floodmodule(Y,L,\gamma_{Y})\right)\leq\varepsilon\max\left(1,\frac{2\ell_K}{h^{d+1}}\right).\]
  If $\gamma_X$ and $\gamma_Y$ are obtained by DTM with parameter $m\in(0,1]$, then:
    \[\interl\left(\floodmodule(X,L,\gamma_X), \floodmodule(Y,L,\gamma_{Y})\right)\leq\varepsilon\left(1+\frac{1}{\sqrt{m}}\right).\] 
\end{theorem}
Thus, if the input $X$ is slightly perturbed while $L$ remains fixed, the associated persistence module changes only slightly (with respect to the interleaving distance), as do its vectorizations, whether they are based on the fibered barcode or finer representations like SCDR.

\subsection{Computational approach}
\label{sec:computational_approach}

\paragraph{Landmark selection}
As in~\cite{graf2025floodcomplex}, we want to minimize $\hausdo(X,L)$ to ensure a well-distributed set of landmarks.
However, given a fixed number of landmarks $|L|$, finding the optimal $L\subset X$ that minimizes $\hausdo(X,L)$ amounts to solving a metric $|L|$-center problem which is \np-hard~\cite{hochbaum1985best}.
Instead, we adopt below the same greedy strategy as in~\cite{graf2025floodcomplex}, based on furthest point sampling.
\changed{
  See \cref{sec:landmarks_sampling,sec:landmarks_number} for discussions on the sampling strategy and the number of landmarks.
}

\paragraph{Masking strategy}
A key aspect of the computational efficiency of the Flood filtration is its local nature:
the filtration value $\fflood(\sigma)$ of a simplex $\sigma\in\del(L)$ is not affected by points that are sufficiently far from $\sigma$.
More precisely, if $\ball_r(c)$ is an enclosing ball of $\sigma$, then $\fflood(\sigma)=\max_{p\in\conv\sigma}\dist(p,X\cap\ball_{\sqrt{2}r}(c))$. This allows for a masking strategy in which each simplex only considers the points of $X$ that are in the $\sqrt{2}$-dilation of its minimal enclosing ball (meb).
Unfortunately, this approach does not directly extend to the sublevel Flood bifiltration, since a point with a low $\gamma$ value may influence the set of bigrades of arbitrarily far simplices.
To keep the number of points per simplex manageable, we resort to the following approximation.
We first compute the minimal set of bigrades $\bigrades(\{l\})$ of each vertex $l\in L$, and write $X^{\{l\}}\subseteq X$ for the points that contribute to it.\footnote{i.e., $X^{\{l\}}=\left\{x\in X\mid\nexists y\in X\text{ s.t. } y\neq x \text{ and } \gamma(y)\leq\gamma(x) \text{ and } \dist(y,l)\leq\dist(x,l)\right\}$.}
Now, let $\sigma\in\del(L)$. We define the set of points that are considered for determining its bigrades as
\[X^\sigma=\bigcup_{\tau\in\cell(\sigma)}\left(X\cap\ball_{\sqrt{2}r_\tau}(c_\tau)\right)\cup\bigcup_{l\in\vstar(\sigma)}X^{\{l\}},\]
where $\cell(\sigma)$ contains the maximal cofaces of $\sigma$,
$\ball_{r_\tau}(c_\tau)$ is the minimal enclosing ball of $\tau$,
and $\vstar(\sigma)$ contains the vertices appearing in the cofaces of $\sigma$.\footnote{i.e., $\vstar(\sigma)=\{l\in L\mid\exists \tau\text{ s.t. }\sigma\subset\tau\text{ and } l\in\tau\}$.}
Informally, the left union includes points that are \textit{near} $\sigma$, which impact the set of bigrades $\bigrades(\sigma)$ for low diameters and large $\gamma$ values. The right union, on the other hand, contains \textit{farther points} that typically impact $\bigrades(\sigma)$ for large diameters and low $\gamma$ values.
We write:
\begin{equation}
  \cflood_{r,s}(X,L,\gamma)=\{\sigma\in\del(L)\mid\conv(\sigma)\subset\bigcup\limits_{x\in X^\sigma_s}\ball_r(x)\} \text{ with } X^\sigma_s=\{x\in X^\sigma\mid\gamma(x)\leq s\}.
  \label{eq:function_flood_masked}
\end{equation}

\begin{theorem}
  \label{th:computation}
  $\cflood(X,L,\gamma)$ is a bifiltration and, in addition, noting $\Delta(L)=\max\limits_{\sigma\in\del(L)}\delta(\sigma)$ where $\delta$ is the diameter, we have
  $\interl(\floodmodule(X,L,\gamma), \widetilde{\floodmodule}(X,L,\gamma))\leq2\Delta(L)$.
\end{theorem}
While this result suggests that a dense set of landmarks is needed to benefit from the properties of the Flood bifiltration with our approximate masking strategy, this bound is overly pessimistic, because it does not account for the inclusion of all \textit{near} points in $X^\sigma$.
%See \cref{fig:approximation,fig:stability} for a practical evaluation of its stability and its ability to approximate the sublevel offset bifiltration.
See \cref{fig:approximation} for a practical evaluation of the ability of this masked Flood bifiltration to approximate the sublevel offset bifiltration, and \cref{fig:stability} for an evaluation of its stability with respect to perturbations of the input.
See also \cref{sec:masking,table:masking_evaluation} for a comparison between the exact (unmasked, \cref{eq:function_flood}) and approximate (masked, \cref{eq:function_flood_masked}) sublevel Flood bifiltrations in terms of approximation and runtime.

% To do so, each simplex $\sigma\in\del(L)$ considers two kinds of points in $X$:
% \begin{itemize}
%   \item \textit{near points}: the points within $X\cap\ball_{\sqrt{2}r}(c)$ where $\ball_{r}(c)$ is the meb of $\sigma$. Informally, these points typically impact the set of bigrades $\bigrades(\sigma)$ for low diameters and high $\gamma$ values;
%   \item \textit{far points}: the critical points of each of its vertices $l\in\sigma$, i.e., the points of $X$ that contribute to one of the sets of bigrades $\phi(\{l\})$. Conversely, these points typically impact $\bigrades(\sigma)$ for high diameters and low $\gamma$ values.
% \end{itemize}

\paragraph{Simplex discretization and computation}

As in the original Flood algorithm, we discretize each convex hull $\conv(\sigma)$ by a discrete point set $P_\sigma\subset\conv(\sigma)$, which is chosen on a grid obtained with evenly spaced barycentric coordinates.
Then, when computing the set of bigrades of $\sigma\in\del(L)$ (see \cref{eq:true_bigrades}),
$\max_{p\in\conv\sigma}\dist(p,X^\sigma_s)$
is replaced by 
$\max_{p\in P_\sigma}\dist(p,X^\sigma_s)$.
The error introduced by this approximation is controlled by the number of grid points on each simplex edge (see \cite{graf2025floodcomplex}, Theorem 5 and Lemma 6, for a detailed discussion).
Thus, with $X^\sigma$ still ordered by increasing value of $\gamma$, the set of bigrades $\phi(\sigma)$ is computed as the Pareto front of:
\[\left\{\left(\max_{p\in P_\sigma}\min_{x\in \Xincr{i}^\sigma}\dist(p,x),\gamma(X_i^\sigma)\right),1\leq i\leq|X^\sigma|\right\}.\]
Concretely, given the distance matrix $D_\sigma=\dist(P_\sigma,X^\sigma)$, one can compute $\phi(\sigma)$ by first applying a cumulative minimum operation along each row, and then taking the maximum on each column. We implemented this as batched operations on GPU tensors.
\changed{See \cref{fig:sublevel_flood_computation} in \cref{sec:comp} for an illustration.}
The speedup obtained from GPU acceleration is discussed in \cref{sec:gpu} and illustrated in \cref{fig:cpu_vs_gpu}.

\section{Experiments}
\label{sec_experiments}

We implemented our algorithm in Python using PyTorch~\cite{pytorch} to compute bifiltration values with tensors on the GPU.
For comparisons, we used the multipers library~\cite{multipers} for the multiparameter module approximation~\cite{loiseaux2025multi}, vectorizations~\cite{loiseaux2023framework,vipond2020multiparameter}, and wrappers of multiparameter PH algorithms (sublevel Delaunay bifiltrations~\cite{alonso2024delaunay} and minimal presentations~\cite{kerber2021fast}).
Diameter-based and density-based persistence diagrams were computed and vectorized with the Gudhi library~\cite{maria2014gudhi}.
Finally, the machine learning pipelines were set up using the scikit-learn library~\cite{pedregosa2011scikit}.
In all our experiments, we sample 20 points on each simplex edge to define the barycentric grids $P_\sigma$ (see \cref{sec:computational_approach}; an edge has 20 samples, a triangle 210 and a tetrahedron 1540).
Timing results were obtained on a server with 12 CPU cores, 64 GB of RAM and an H100 GPU.
See \cref{sec_additionalExperiments} for additional experiments.

\subsection{Runtime}

In \cref{fig:timings}, we compare the running times for computing the bifiltration plus its MMA, with an optional intermediate minimal presentation (MP) step for the sublevel Delaunay and Rips bifiltrations.
For our sublevel Flood bifiltration, we test several numbers of landmarks $|L|\in\{100,300,500,1000\}$. 
We run these experiments for $d=2$ and $d=3$: in both cases, half of the points in $X\subset\bbr^d$ are distributed over the unit $(d-1)$-sphere with added Gaussian noise, while the other half are uniformly distributed inside the $d$-ball of radius $3/2$.
We show the running time of the bifiltration construction alone (without MMA) in \cref{fig:timings_additional}.
We also report in \cref{fig:structure_sizes}, the sizes of the constructions: total number of bigrades (which coincides with the number of simplices for 1-critical bifiltrations); number of generators after minimal presentation; number of summands after MMA.
This shows that our construction offers a clear speedup for computing an MMA from a point set compared to the Delaunay bifiltration, when $X$ is sufficiently large.
In \cref{fig:timings},
% for 500 landmarks, the Flood bifiltration becomes faster than the Delaunay bifiltration starting from $|X|\simeq1,000$.
\changed{using the Flood bifiltration becomes worthwhile starting from $|X|$ between 200 and 3,000 depending on the dimension $d$ and the number of landmarks.}
As $|X|$ continues to grow, the computation time scales more favorably than for other methods: this is due to the fixed size of the complex on which the bifiltration is built.

\begin{figure}
  \includegraphics[width=\linewidth]{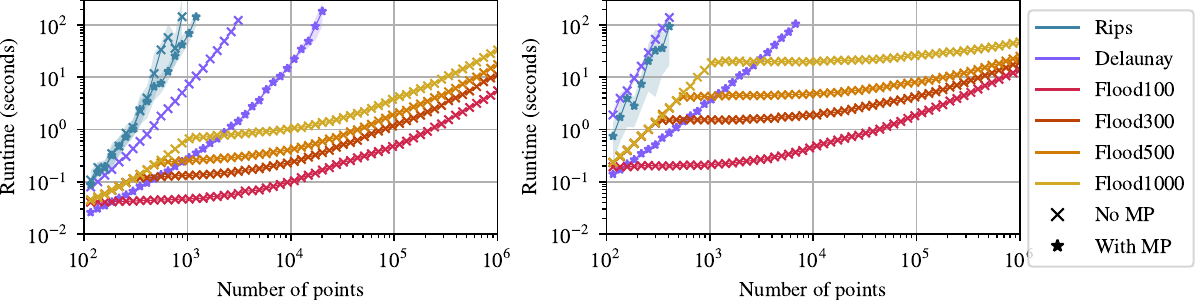}
  \caption{Running time of the bifiltration computation and its multiparameter module approximation. The homology is computed for degrees $\hdim\in\{0,1\}$ in $\bbr^2$ (left) and degrees $\hdim\in\{0,1,2\}$ in $\bbr^3$ (right).}
  \label{fig:timings}
\end{figure}

\subsection{Approximation and stability}

\cref{fig:approximation} assesses the ability of the sublevel Flood bifiltration to estimate the topology of the reference sublevel offset bifiltration, by reporting the matching distance (\cref{sec_additional_background}) $\matchd\left(\homol{\hdim}(\flood(X,L,\gamma)), \homol{\hdim}(\alphaf{}(X,\gamma))\right)$ for $\hdim\in\{0,1,2\}$ when the size of $L$ increases until $L=X$.
\begin{figure}
  \centering
  \includegraphics[width=.3295\linewidth]{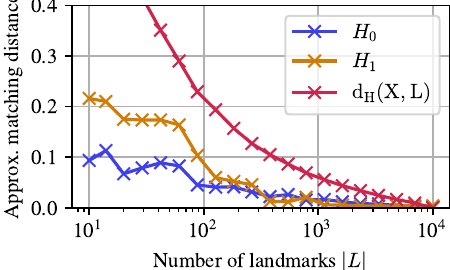}
  \includegraphics[width=.3295\linewidth]{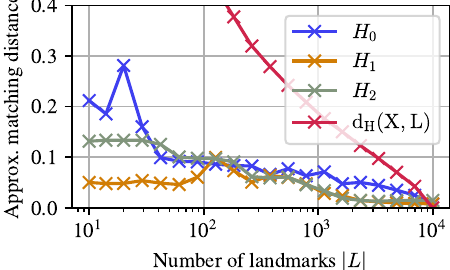}
  \includegraphics[width=.3295\linewidth]{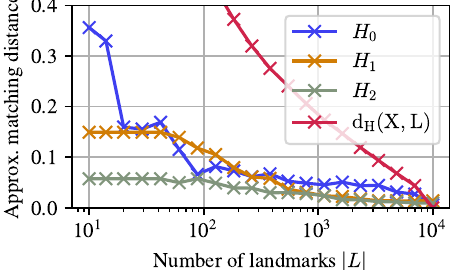}
  \caption{Matching distance between the modules of sublevel Flood and sublevel Delaunay bifiltrations, when the number of landmarks $|L|$ increases.
  $X$ has 10,000 points distributed over the noisy ring in $\bbr^2$ (left), the noisy sphere in $\bbr^3$ (center) and the noisy torus in $\bbr^3$ (right).}
  \label{fig:approximation}
\end{figure}
We also evaluate in \cref{fig:stability} the empirical stability of the construction with respect to the input $X$ by computing $\matchd\left(\floodmodule(X,L_X,\gamma_X), \floodmodule(Y,L_Y,\gamma_Y)\right)$ with $\lVert x_i-y_i\rVert\leq\varepsilon$, both when $L_X=L_Y$ is fixed and when $L_X$ and $L_Y$ are computed independently from $X$ and $Y$.

% \subsection{Datasets}

% The datasets used in Flood \cite{graf2025floodcomplex}:
% \begin{itemize}
%   \item Synthetic datasets 
%   \begin{itemize}
%     \item \texttt{swiss\_cheese}: uniform points in unit cube, with $k\in\{10,20\}$ random balls removed
%     \item \texttt{rocks}: synthetic porous materials with the \texttt{PoreSpy} library~\cite{gostick2019porespy}
%   \end{itemize}
%   \item Real-world datasets 
%   \begin{itemize}
%     \item \texttt{modelnet10}: simple objects (few topological features)~\cite{wu20153d}
%     \item \texttt{mcb}: more geometrically and topologically interesting features~\cite{kim2020large}
%     \item \texttt{coral}: point sampled on surfaces meshes of corals \url{https://3d.si.edu/corals}
%   \end{itemize}
% \end{itemize}

\subsection{Classification problems}

To demonstrate the practical value of the sublevel Flood bifiltration% and its ability to capture meaningful topological features
, we compare its performance against other bifiltrations and single-parameter homology in classification tasks.
The specifics of our datasets are provided in \cref{sec_dataDescription}.

For bifiltrations, we vectorize the multiparameter module approximation (MMA) into either stable candidate decomposition representations (SCDR)~\cite{loiseaux2023framework} or multiparameter persistence landscapes (MPL)~\cite{vipond2020multiparameter}.
To evaluate the interest of the multiparameter framework, we also compare to single-parameter filtrations: a diameter-based filtration (Delaunay--\cech) and a density-based filtration (sublevel filtration of $\gamma$ evaluated on a regular cubical complex in the ambient space). The obtained persistence diagrams are vectorized
using the ATOL method~\cite{royer2021atol} (respectively "diameterPD" and "densityPD"). We also consider the concatenation of both vectorizations ("bothPD").
All vectorizations are given to a LightGBM~\cite{ke2017lightgbm} classifier with default parameters. 
Hyperparameters are validated with 5-fold cross-validation, using the Bayesian optimization framework Optuna~\cite{akiba2019optuna} with 50 trials.
For SCDR, we cross-validate the power $p\in\{0,1\}$, the bandwidth $h\in[10^{-3},10^0]$ on a log-scale, and the resolution in \{$10\times10, 20\times20, 50\times50$\}.
For MPL, only the resolution is cross-validated in the same range.
For ATOL, the number of $k$-means centroids is selected in $\{32,64,128\}$ and the contrast function (\cite{royer2021atol}, Eq. 1) is selected in $\{\text{Gaussian},\text{Laplacian},\text{indicator}\}$.

\paragraph{Toy synthetic dataset (\cref{fig:synthetic_scores}, left, and \cref{table:synthetic_timings})}
This dataset contains point sets with lower-density regions of varying sizes. The goal is to classify the point sets based on the diameter of these lower-density areas, where a difficulty parameter controls how well-separated the classes are. On this task, 2-parameter persistence methods outperform single-parameter ones, for all difficulty parameters.
%In particular, they outperform the concatenation of both diameter- and density-based persistence diagrams,
This suggests that they capture meaningful topological information that is not captured by single-parameter PH.
Besides, except at the highest difficulty level, the gap in accuracy score between Delaunay and Flood bifiltrations is negligible, and the number of landmarks has little influence on the score: this confirms that, in this task, the classification accuracy mostly depends on the 2-parameter persistence features that exist for large diameter parameters $r$. These features do not require a particularly fine complex to be captured, as they correspond to the low-density regions.

\begin{figure}
  \centering
  \includegraphics[width=\linewidth]{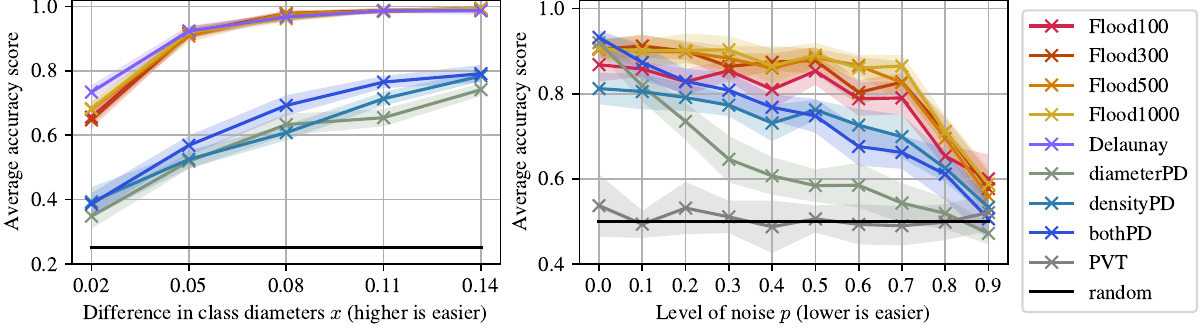}
  \caption{Accuracy scores (averaged over 10 splits) of single-parameter PH (diameterPD, densityPD and bothPD) and multiparameter PH (Flood and Delaunay).
  \textbf{Left:} Synthetic noisy swisscheese toy datasets in $\bbr^2$ parameterized by difficulty $x$.
  Here, the dataset contains 800 sets equally split into four classes characterized by the diameter of their low-density regions. Each set contains 5,000 points.
  \textbf{Right:} Synthetic porous material datasets on a $64^3$ grid.
  The problem difficulty increases with the parameter $p$ that quantifies the level of ambient noise.
  We also show the result obtained with the neural network-based method PVT~\cite{zhang2022pvt} for reference.%following the methodology of \cite{graf2025floodcomplex}.
  }
  \label{fig:synthetic_scores}
\end{figure}

\begin{table}
  \centering
    \caption{Timing in seconds for the noisy swisscheese datasets with 5,000 and 20,000 points (\texttt{NSC-*D-*k}) and the porous rocks dataset (\texttt{rocks64}), averaged over all their difficulties. The left number is the computation time of MMAs of the whole dataset. The right number is the classification time from the MMAs (averaged over the 10 splits). \oom\ means that memory was exceeded. A point set in \texttt{rocks64} has on average 69,000 points for $p=0$ and 130,000 points for $p=0.9$.}
  \begin{tabular}{|l|rrrrr|}
  \hline
    Dataset                 & Flood100 &   Flood300 &   Flood500 &   Flood1000 &   Delaunay \\
  \hline
    \texttt{NSC-2D-5k} & 38.3 / 17.4 & 85.3 / 32.3 & 142.7 / 49.1 & 312.6 / 57.8 & 1484.6 / 91.8  \\
    \texttt{NSC-2D-20k} & 91.3 / 11.5 & 150.2 / 12.9 & 219.3 / 19.4 & 416.4 / 38.3 & \oom \\
    \texttt{NSC-3D-5k} & 183.9 / 24.8 & 800.7 / 43.6 & 2309.3 / 53.0 & 13990.6 / 130.3 & \oom \\
    \texttt{rocks64} & 602.7 / 11.3 & 1086.7 / 16.1 & 1815.3 / 41.8 & 5925.6 / 70.3 &  \oom \\ 
  \hline
      %  38.3 &       85.3 &      142.7 &       312.6 &     1484.6 &          7.5 \\
      %  63.1 &      151.5 &      231.7 &       268.1 &      363.4 &         82.7 \\
      %  17.4 &       32.3 &       49.1 &        57.8 &       91.8 &         15.2 \\
  \end{tabular}
  \label{table:synthetic_timings}
\end{table}

% \begin{figure}
%   \centering
%   \includegraphics[width=.48\linewidth]{figures/swisscheese_2_5000_scores.pdf}
%   \includegraphics[width=.48\linewidth]{figures/swisscheese_2_20000_scores.pdf}
% %   \begin{minipage}{0.48\linewidth}
% %   \includegraphics[width=\linewidth]{figures/classification_2Dtoy_score.pdf}
% %   \end{minipage}
% %   \begin{minipage}{0.48\linewidth}
% %   \begin{tabular}{lr}
% %   \hline
% %   Method       &   Mean runtime (s) \\
% %   \hline
% %   Flood100-MMA &               38.4 \\
% %   Flood300-MMA &               80.7 \\
% %   Flood500-MMA &              126.0 \\
% %   Delaunay-MMA &              607.3 \\
% %   Delaunay-PI  &             (41.1) \\
% %   Density-PI   &                5.1 \\
% %   Both-PI      &             (45.1) \\
% %   \hline
% %   \end{tabular}
% % \end{minipage}
%   \caption{Accuracy scores (averaged on 10 splits) of single-parameter PH (diameterPD, densityPD and bothPD) and multiparameter PH (Flood and Delaunay) on the toy, synthetic noisy swisscheese datasets parameterized by difficulty $x$.
%   The dataset contains here 800 sets equally split in the four classes characterized by the diameter of their low-density regions. Each set contains 5,000 (left) or 20,000 (right) points in $\bbr^2$.
%   The black line shows the score of a random classifier (0.25).} 
% \end{figure}

\paragraph{Porous rocks dataset (\cref{fig:synthetic_scores}, right, and \cref{table:synthetic_timings})} This dataset models porous material generated using various noise generators. The task consists of classifying according to the generator used. The task is rendered more difficult as the added ambient noise increases.
% The second dataset models porous material and is generated with the \texttt{PoreSpy} library~\cite{gostick2019porespy}.
% It generates binary $64^3$ grids that we convert into a point set, with some Gaussian noise (to avoid the generic position due to the lattice structure).
% The porosity is chosen randomly, uniformly in $[0.05,0.95]$.
% The classification task consists in determining which porous material generator was used, between:
% \texttt{blobs} that generates random noise combined with Gaussian blur, with a blobiness parameter chosen uniformly in $[0.5,2]$;
% and \texttt{fractals} which is simplex noise with a frequency chosen uniformly in $[0.05,0.2]$.
% To measure the ability of multiparameter PH to handle noise and outliers, we parameterize the difficulty of the problem with an ambient noise parameter which represents the probability $p$ of each element of the binary grid to be set to 1 (after the binary generation and before the point set conversion). We only use $H_2$.
% Results are shown in \cref{fig:synthetic_scores}, right.
Again, the Flood bifiltration outperforms the single-parameter methods when the level of noise is high ($p\geq0.2$). Here, the number of landmarks has more impact on the accuracy score: this suggests that generator classification is easier when low-diameter 2-parameter persistence features are captured.

% \begin{figure}
%   \centering
%   \includegraphics[width=.48\linewidth]{figures/rocks_64_scores.pdf}
%   \caption{Accuracy scores (averaged on 10 splits) of single-parameter PH (diameterPD, densityPD and bothPD) and multiparameter PH (Flood and Delaunay) on the synthetic porous material datasets on a $64^3$ grid.
%   The problem difficulty decreases with the parameter $x$ that quantifies the level of ambient noise.
%   The black line shows the score of a random classifier (0.5)}
% \end{figure}

\paragraph{Time series datasets}
\begin{wrapfigure}[9]{r}{0.35\textwidth}
  \vspace{-1\baselineskip}
  \includegraphics[width=\linewidth]{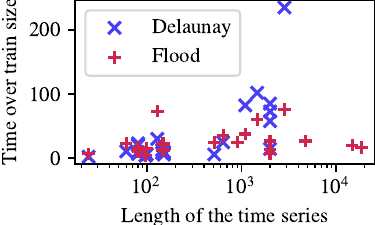}
\end{wrapfigure}
Finally, we conducted classification experiments on the time series UCR datasets~\cite{dau2019ucr} using time-delay embeddings. \cref{table:ucr_scores} (and \cref{table:ucr_scores_additional} in \cref{sec_additionalExperiments}) show that, generally speaking, the sublevel Flood bifiltration is on par with the Delaunay bifiltration in terms of classification accuracy.
Its key strength rather lies in enabling the use of larger embeddings (i.e., longer
time series): see the out-of-memory (OOM) errors encountered with Delaunay on \texttt{KLC}, \texttt{CD}, \texttt{BH}.
The inlined figure shows the classification time from the MMA (averaged over the cross-validation trials, and divided by the training dataset size) and illustrates the improved scalability of our approach.

\begin{table}
  \centering
  \caption{Accuracy scores on UCR time series datasets. $\deln$ and $\flood$ denote the Delaunay and Flood bifiltrations, respectively.}
  %\begin{tabular}{|l||rrr|r|rrrr|}
  \begin{tabular}{|l||rr|r|rrrr|}
  \hline
  % Dataset & ED & DTWl & DTWc & \cech\ & $\deln$/MPL & $\flood$/MPL & $\deln$/SCDR & $\flood$/SCDR \\
  Dataset & ED & DTWc & \cech\ & $\deln$/MPL & $\flood$/MPL & $\deln$/SCDR & $\flood$/SCDR \\
  \hline
\texttt{DPOAG} & 0.626 & \textbf{0.770} & 0.698 & 0.741 & 0.669 & 0.698 & 0.719\\
\texttt{DPOC} & 0.717 & 0.717 & 0.699 & 0.641 & 0.649 & 0.659 & \textbf{0.732}\\
\texttt{DPTW} & 0.633 & 0.590 & \textbf{0.647} & 0.576 & \textbf{0.647} & 0.597 & 0.640\\
\texttt{IPD} & \textbf{0.955} & 0.950 & 0.710 & 0.722 & 0.724 & 0.747 & 0.724\\
\texttt{PC} & \textbf{0.933} & 0.878 & 0.889 & 0.744 & 0.728 & 0.883 & 0.911\\
\texttt{LN2} & 0.754 & \textbf{0.869} & 0.721 & 0.525 & 0.541 & 0.607 & 0.705\\
\texttt{Haptics} & 0.370 & \textbf{0.377} & 0.367 & 0.192 & 0.237 & 0.347 & 0.198\\
\texttt{ACSF1} & 0.540 & 0.640 & 0.740 & 0.380 & 0.380 & 0.400 & \textbf{0.750}\\
\texttt{PigCVP} & 0.082 & 0.154 & 0.139 & 0.019 & 0.019 & \textbf{0.226} & 0.207\\
\texttt{PArtP} & 0.125 & 0.245 & 0.361 & 0.029 & 0.024 & \textbf{0.635} & 0.433\\
\texttt{Rock} & \textbf{0.840} & 0.600 & 0.420 & 0.420 & 0.420 & 0.420 & 0.420\\
\texttt{KLC} & 0.278 & 0.815 & \textbf{0.835} & \oom & 0.145 & \oom & 0.802\\
\texttt{CD} & 0.445 & 0.610 & 0.555 & \oom & 0.555 & \oom & \textbf{0.640}\\
\texttt{BH} & \textbf{0.746} & 0.610 & 0.737 & \oom & 0.717 & \oom & 0.737\\
  \hline
  \end{tabular}
  \label{table:ucr_scores}
\end{table}

% \begin{figure}
%   \centering
%   \includegraphics[width=.4\linewidth]{UCR_MMA_time.pdf}
%   \includegraphics[width=.4\linewidth]{UCR_classif_time.pdf}
%   \caption{Timings on the UCR data according to the length of the time series.
%   \textbf{Left: } MMA computation time (averaged on the whole dataset)
%   \textbf{Right: } Classification time (averaged on the cross-validation trials, and divided by the train dataset size).}
% \end{figure}

\section{Conclusion and future work}
\label{sec_conclusion}

In this work, we introduced the \textit{sublevel Flood bifiltration}, an extension of the Flood filtration to the context of sublevel-based multiparameter persistent homology.
It provides a practical approximation of the reference sublevel Delaunay bifiltration while being more scalable.

Some limitations of this work are inherent to the landmark-based approach. Indeed, the construction will miss the topological features that are smaller than the resolution of $\del(L)$. Besides, relying on $\del(L)$ limits the ambient dimension that we are able to consider, and leads to instabilities when the combinatorics of $\del(L)$ changes.
As future work, we plan to investigate more deeply the choice of landmarks, for instance with density-aware approaches. %, and to improve the coarsening strategy with a procedure that removes carefully bigrades when it reduces the less the accuracy.
Finally, we think that there remains room for improving the theoretical guarantees of both Flood filtration and sublevel Flood bifiltration, to explain, possibly with more assumptions, why this approach provides far better results than subsampling in many cases.
% \paragraph{Limitations}
% \begin{itemize}
%   \item inherent to the method: resolution limited by the resolution of the $\del(L)$ (miss to small topological features); limited in ambiant dimension to be able to compute $\del(L)$; instability with respect to $L$ when the combinatorics of $\del(L)$ changes
%   \item investigate more deeply the choice of landmarks (density-aware sampling~\cite{climaco2025density}) 
%   \item theoretical improvement: better approximation results in specific cases (first, on Flood filtration)
%   \item improve coarsening procedure for a better accuracy
% \end{itemize}

\section*{Acknowledgments}
This work is partially supported by the European Commission grant ERC-2019-COG \emph{``TORI''} (ref. 863464, \url{https://erc-tori.github.io/}) and by ANR-23-PEIA-0004 (PDE-AI, \url{https://pde-ai.math.cnrs.fr/}).
We gratefully acknowledge support from the CNRS/IN2P3 Computing Center (Lyon, France) for providing computing and data-processing resources needed for this work.

\bibliographystyle{abbrv}
\bibliography{biblio}

\newpage
\hrule height 4pt
\vskip 0.25in
\begin{center}
{\LARGE\bfseries Supplementary material\par}
\end{center}
\vskip 0.29in
\vskip -\parskip
\hrule height 1pt
\vskip 0.09in

\appendix
\section{Additional background}
\label{sec_additional_background}

In this section, we provide formal definitions for the multiparameter persistence modules and the distances between them that are used in the paper. 
We write $\preceq$ for the partial order on $\bbr^n$, i.e., $\alpha\preceq\beta$ if and only if $\alpha_i\leq\beta_i$ for all $i\in\llbracket1,n\rrbracket$.

\paragraph{Multiparameter persistence module} A \textit{multiparameter persistence module} with $n$ parameters is an indexed family of vector spaces $M=(M_\alpha)_{\alpha\in\bbr^n}$, with a family of \textit{transition maps} $\varphi_{\alpha,\beta}:M_\alpha\to M_\beta$ for all $\alpha\preceq\beta$, that satisfy $\varphi_{\alpha,\gamma}=\varphi_{\beta,\gamma}\circ\varphi_{\alpha,\beta}$ for all $\alpha\preceq\beta\preceq\gamma$.

Often, a multiparameter persistence module is obtained by applying the homology functor $\homol{\hdim}$ on a multifiltration $(\calk_\alpha)_{\alpha\in\bbr^n}$, giving $M=(\homol{\hdim}(\calk_\alpha))_{\alpha\in\bbr^n}$. The linear maps $\varphi_{\alpha,\beta}$ are induced by the simplicial inclusions $\calk_\alpha\subseteq\calk_\beta$ for all $\alpha\preceq\beta$.

\paragraph{Interleaving distance} The \textit{interleaving distance}~\cite{lesnick2015theory} allows multiparameter persistence modules to be compared. If $M$ is a multiparameter persistence module, its $\varepsilon$-shift is the module $M_\varepsilon=(M_{\alpha+\vec{\varepsilon}})_{\alpha\in\bbr^n}$ with $\vec{\varepsilon}=(\varepsilon,\ldots,\varepsilon)\in\bbr^n$.
Two multiparameter persistence modules $M$ and $N$ (of transition maps $\varphi$ and $\psi$ respectively) are $\varepsilon$-interleaved if there exist $f:M\to N_\varepsilon$ and $g:N\to M_\varepsilon$ such that $g_{\cdot+\varepsilon}\circ f_\cdot=\varphi_{\cdot,\cdot+2\varepsilon}$ and $f_{\cdot+\varepsilon}\circ g_\cdot=\psi_{\cdot,\cdot+2\varepsilon}$.
Then, the interleaving distance between two multiparameter persistence modules $M$ and $N$ is defined as
\[\interl(M,N)=\inf\{\varepsilon>0\mid \text{$M$ and $N$ are $\varepsilon$-interleaved}\}.\]

\paragraph{Matching distance}
\changed{Computing the interleaving distance (and even any approximation by a factor smaller than 3) is \np-hard~\cite{bjerkevik2020computing}.}
The \textit{matching distance}~\cite{cerri2013betti} between two 2-parameter persistence modules $M$ and $N$ is defined over the associated fibered barcodes, i.e., the collection of (single-parameter) persistence barcodes obtained when restricting a module to one-dimensional slices $\ell$ defined by a line of equation $y=ax+b$ (giving the single-parameter persistence module $M_\ell=\left(M_{x,ax+b}\right)_{x>0}$).
For each slice $\ell$, we define the following distance along $\ell$:
\begin{equation}
\dist_\mathrm{slice}^{M,N}(\ell)=w(\ell)\bottle(M_\ell, N_\ell)\text{ with }w(\ell)=\left\{\begin{array}{ll}
  \frac{1}{\sqrt{1+a^2}} & \text{if } a\geq1\\
  \frac{1}{\sqrt{1+1/a^2}} & \text{if } 0<a<1
\end{array}\right.
\end{equation}
where $\bottle$ is the single-parameter bottleneck distance.
Then, the matching distance is defined as
\begin{equation}
\matchd(M,N)=\sup_{\ell}\dist_\mathrm{slice}^{M,N}(\ell)
\end{equation}
where the $\sup$ is over the slices with $a>0$. The correcting weights $w(\ell)$ are here to ensure that $\matchd(M,N)\leq\interl(M,N)$.
Both exact~\cite{kerber2019exact,bjerkevik2023asymptotic} and approximate~\cite{kerber2020efficient} algorithms have been proposed to compute the matching distance.

\begin{figure}
	\centering
  \begin{overpic}[width=.4\linewidth]{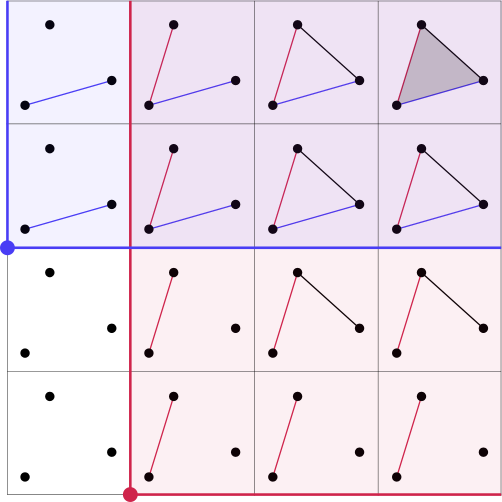}
    \put(12.5,-5){0}
    \put(37.5,-5){1}
    \put(62.5,-5){2}
    \put(87.5,-5){3}
    \put(-5,12.5){0}
    \put(-5,37.5){1}
    \put(-5,62.5){2}
    \put(-5,87.5){3}
  \end{overpic}
	\hspace{1cm}
	\begin{overpic}[width=.4\linewidth]{kcriticality.png}
    \put(12.5,-5){0}
    \put(37.5,-5){1}
    \put(62.5,-5){2}
    \put(87.5,-5){3}
    \put(-5,12.5){0}
    \put(-5,37.5){1}
    \put(-5,62.5){2}
    \put(-5,87.5){3}
  \end{overpic}
  \vspace{.2cm}
	\caption{Two bifiltrations over a filled triangle indexed over $\llbracket0,3\rrbracket^2$. The left one is 1-critical while the right one is 3-critical. The blue (resp. red) dots represent the minimal set of bigrades $\bigrades(\sigma)$ of the blue (resp. red) edge. Their respective supports $\supp(\sigma)$ are highlighted with the same color.}
  \label{fig:bifiltration_criticality}
\end{figure}

\section{Proofs}
\label{sec_proofs}

This section contains the proofs of the theorems stated in \cref{sec:properties,sec:computational_approach}.

\begin{proof}[\textbf{Proof of \cref{th:stability_X_gamma}}]
  Let $\sigma\in\flood_{r,s}(X,L,\gamma_X)$ and $p\in\conv\sigma$. There is an $x\in X$ such that $\dist(p,x)\leq r$ and $\gamma_X(x)\leq s$.
  Then, there exists $y\in Y$ such that $\dist(x,y)\leq\hausdo(X,Y)$. Therefore $\dist(p,y)\leq\dist(p,x)+\dist(x,y)\leq r+\hausdo(X,Y)$.
  Besides:
  \begin{align*}
    \gamma_Y(y)-s&\leq\gamma_Y(y)-\gamma_X(x)\\
    &\leq|\gamma_Y(y)-\gamma_X(x)|\\
    &\leq|\gamma_Y(y)-\gamma_Y(x)|+|\gamma_Y(x)-\gamma_X(x)|\\
    &\leq\lipconst\hausdo(X,Y)+\lVert\gamma_X-\gamma_Y\rVert_\infty
  \end{align*}
  This means that $\sigma\in\flood_{r+\hausdo(X,Y),s+\lipconst\hausdo(X,Y)+\lVert\gamma_X-\gamma_Y\rVert_\infty}(Y,L,\gamma_Y)$. By writing $\alpha=\max\left(\hausdo(X,Y),\lipconst\hausdo(X,Y)+\lVert\gamma_X-\gamma_Y\rVert_\infty\right)$, we have:
  \begin{align*}
    \flood_{r,s}(X,L,\gamma_X)&\subseteq\flood_{r+\alpha,s+\alpha}(Y,L,\gamma_Y)\\
    \flood_{r,s}(Y,L,\gamma_Y)&\subseteq\flood_{r+\alpha,s+\alpha}(X,L,\gamma_X)
  \end{align*}
  These simplicial inclusions provide an $\alpha$-interleaving between the two modules $\floodmodule(X,L,\gamma_X)$ and $\floodmodule(Y,L,\gamma_{Y})$, hence the result.
\end{proof}

\begin{proof}[\textbf{Proof of \cref{th:stability_X}}]
  Under these conditions we have $\hausdo(X,Y)\leq\varepsilon$. Besides:
  \begin{itemize}
    \item if $\gamma$ is obtained by KDE with an $\lipconst_K$-Lipschitz kernel $K$ and bandwidth $h$, then both $\gamma_X$ and $\gamma_Y$ are $\frac{\lipconst_K}{h^{d+1}}$-Lipschitz and one easily checks that $\lVert\gamma_X-\gamma_Y\rVert_\infty\leq\varepsilon\frac{\lipconst_K}{h^{d+1}}$;
    \item if $\gamma$ is obtained by DTM with parameter $m\in(0,1]$, then both $\gamma_X$ and $\gamma_Y$ are 1-Lipschitz, and besides (see, e.g., Theorem 3 of~\cite{chazal2018robust} for the properties of DTM):
    \[\lVert\gamma_X-\gamma_Y\rVert_\infty\leq\frac{1}{\sqrt{m}}\wasser{2}(X,Y)\leq\varepsilon\frac{1}{\sqrt{m}}.\]
  \end{itemize}
  \cref{th:stability_X_gamma} allows us to obtain the upper bound on the interleaving distance between both modules in the two cases.
\end{proof}

%\begin{proof}[\textbf{Proof of \cref{th:approximation}}]
%  \remark{to be proved}
%  \begin{align*}
%     \interl\left(\floodmodule(X,L,\gamma), \offsetmodule(X,\gamma)\right) \leq
%       & \interl\left(\floodmodule(X,L,\gamma), \floodmodule(L,L,\gamma)\right) \quad\big\}\leq\max(1,\lipconst)\hausdo(X,L)\\
%     + & \interl\left(\floodmodule(L,L,\gamma), \offsetmodule(L,\gamma)\right) \\
%     + & \interl\left(\offsetmodule(L,\gamma), \offsetmodule(X,\gamma)\right) \hspace{1.1cm}\big\}\leq\max(1,\lipconst)\hausdo(X,L)
%  \end{align*}
%  Problem:
%  \[\interl\left(\floodmodule(L,L,\gamma), \offsetmodule(L,\gamma)\right) ???\]
%  What would be nice: $\flood_r(L_s^\gamma,L)\simeq\flood_r(L_s^\gamma,L_s^\gamma)$ using $L_s^\gamma\subset L$.
%  In this case
%  \begin{align*}
%    \flood_{r,s}(L,L,\gamma)&=\flood_r(L_s^\gamma,L)\\
%    &\simeq\flood_r(L_s^\gamma,L_s^\gamma)\\
%    &\simeq\offsetf{r}(L_s^\gamma)\text{ (Theorem 3, \cite{graf2025floodcomplex})}\\
%    &=\offsetf{r,s}(L,\gamma)
%  \end{align*}
%\end{proof}

\begin{proof}[\textbf{Proof of \cref{th:computation}}]
  Recall that $\cflood_{r,s}(X,L,\gamma)=\left\{\sigma\in\del(L)\mid\conv(\sigma)\subset\bigcup\limits_{x\in X^\sigma_s}\ball_r(x)\right\}$ with
  \[X^\sigma=\bigcup_{\tau\in\cell(\sigma)}\left(X\cap\ball_{\sqrt{2}r_\tau}(c_\tau)\right)\cup\bigcup_{l\in\vstar(\sigma)}X^{\{l\}}\quad\text{ and }\quad X^\sigma_s=\{x\in X^\sigma\mid\gamma(x)\leq s\}.\]

  We first show that $\cflood_{r,s}(X,L,\gamma)$ is a simplicial complex.
  Let $\sigma\in\cflood_{r,s}(X,L,\gamma)$ and $\tau\subset\sigma$.
  We have $\cell(\sigma)\subset\cell(\tau)$ and $\vstar(\sigma)\subset\vstar(\tau)$ because the cofaces of $\sigma$ are also cofaces of $\tau$. 
  Therefore, $X^\sigma\subset X^\tau$, hence $X^\sigma_s\subset X^\tau_s$.
  This implies that $\conv(\tau)\subset\conv(\sigma)\subset\bigcup\limits_{x\in X^\sigma_s}\ball_r(x)\subset\bigcup\limits_{x\in X^\tau_s}\ball_r(x)$, hence $\tau\in\cflood_{r,s}(X,L,\gamma)$.
  In addition, $(\cflood_{r,s}(X,L,\gamma))_{(r,s)\in\bbr^2}$ is clearly increasing with respect to the partial order $\preceq$ in $\bbr^2$.
  It is therefore a bifiltration.

  As $X^\sigma_s\subset X_s$ for all $s$ and $\sigma$, we have $\cflood_{r,s}(X,L,\gamma)\subseteq\flood_{r,s}(X,L,\gamma)$.
  Conversely, let $\sigma\in\flood_{r,s}(X,L,\gamma)$ and $p\in\conv\sigma$.
  There is an $x\in X$ such that $\dist(p,x)\leq r$ and $\gamma(x)\leq s$.
  If $x\in X^\sigma$, then we directly have $p\in\bigcup\limits_{x\in X^\sigma_s}\ball_r(x)$.
  If $x\not\in X^\sigma$, this means in particular that for all $l\in\sigma$, $x\not\in X^{\{l\}}$.
  Recall that $X^{\{l\}}$ contains the points that contribute to the minimal set of bigrades of $\{l\}$.
  Therefore, for each $l\in\sigma$, there exists an $x_l\in X^{\{l\}}\subset X^\sigma$ such that $\dist(x_l,l)\leq\dist(x,l)$ and $\gamma(x_l)\leq\gamma(x)\leq s$ (hence $x_l\in X^\sigma_s$).
  We have:
  \begin{align*}
    \dist(p,X^\sigma_s)&\leq\min_{l\in\sigma}\dist(p,x_l)\\
    &\leq\min_{l\in\sigma}\left(\dist(p,l)+\dist(x_l,l)\right)\\
    &\leq\min_{l\in\sigma}\left(\dist(p,l)+\dist(x,l)\right)\\
    &\leq\min_{l\in\sigma}\left(\dist(p,l)+\dist(x,p)+\dist(p,l)\right)\\
    &\leq2\min_{l\in\sigma}\left(\dist(p,l)\right)+r\\
    &\leq2\Delta(L)+r
  \end{align*}
  Hence $p\in\bigcup\limits_{x\in X^\sigma_s}\ball_{r+2\Delta(L)}(x)$.
  Therefore, $\flood_{r,s}(X,L,\gamma)\subseteq\cflood_{r+2\Delta(L),s}(X,L,\gamma)$.
  The two inclusions show that the multiparameter persistence modules of $\flood(X,L,\gamma)$ and $\cflood(X,L,\gamma)$ are $2\Delta(L)$-interleaved, hence the result. 
\end{proof}

\section{Computational approach}
\label{sec:comp}

\changed{
\cref{fig:sublevel_flood_computation} illustrates the batched tensor operations we use to compute the set of bigrades of a simplex in the sublevel Flood bifiltration.
Given the distance matrix $D_\sigma=\dist(P_\sigma,X^\sigma)$, we obtain the (non-minimal) set of bigrades by first applying a cumulative minimum operation along each row, and then taking the maximum on each column.
}

\begin{figure}
  \hspace*{-1mm}
  \raisebox{.5cm}{\includegraphics[width=.25\linewidth]{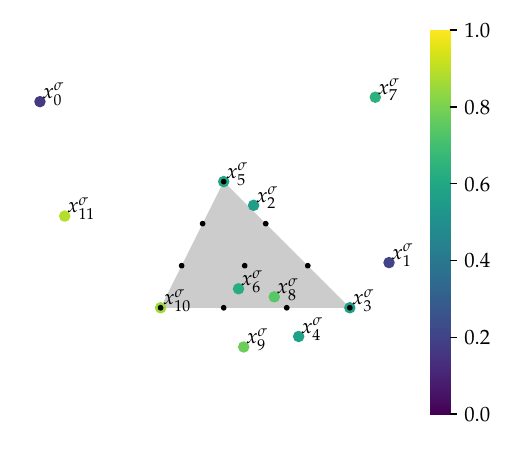}}
  \hspace*{-3mm}
  \raisebox{2cm}{\footnotesize$\overset{D_\sigma}{\longrightarrow}$}
  \hspace*{-2mm}
  \includegraphics[width=.33\linewidth]{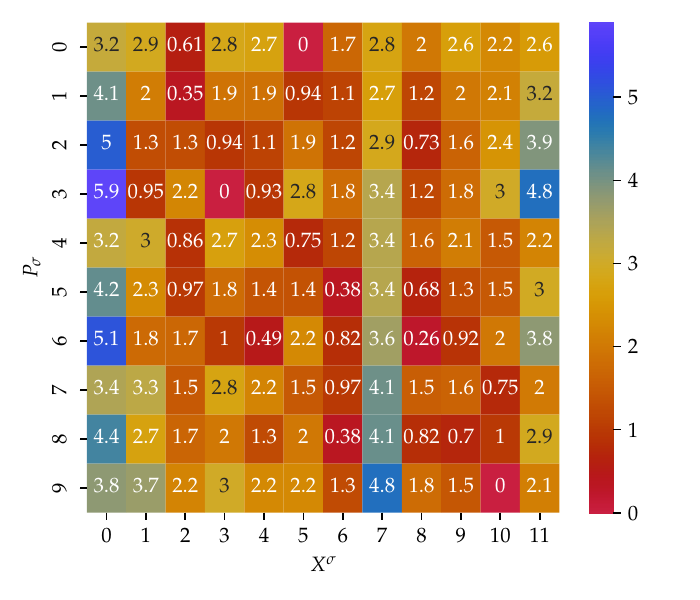}
  \hspace*{-3mm}
  \raisebox{2cm}{\footnotesize$\underset{\min_{x\in \Xincr{i}^\sigma}}{\overset{\text{cumul.}}{\longrightarrow}}$}
  \hspace*{-2mm}
  \includegraphics[width=.33\linewidth]{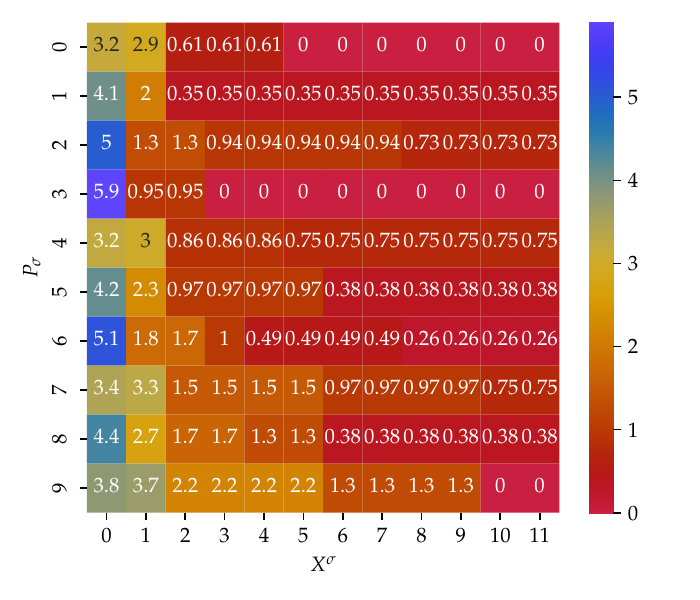}
  \hspace*{-3mm}\\
  \hspace*{3mm}
  \def\svgwidth{\linewidth}\scriptsize%% Creator: Inkscape 1.4.4 (1:1.4.4+202605061436+dcaf3e7d9e), www.inkscape.org
%% PDF/EPS/PS + LaTeX output extension by Johan Engelen, 2010
%% Accompanies image file 'compute.pdf' (pdf, eps, ps)
%%
%% To include the image in your LaTeX document, write
%%   \input{<filename>.pdf_tex}
%%  instead of
%%   \includegraphics{<filename>.pdf}
%% To scale the image, write
%%   \def\svgwidth{<desired width>}
%%   \input{<filename>.pdf_tex}
%%  instead of
%%   \includegraphics[width=<desired width>]{<filename>.pdf}
%%
%% Images with a different path to the parent latex file can
%% be accessed with the `import' package (which may need to be
%% installed) using
%%   \usepackage{import}
%% in the preamble, and then including the image with
%%   \import{<path to file>}{<filename>.pdf_tex}
%% Alternatively, one can specify
%%   \graphicspath{{<path to file>/}}
%% 
%% For more information, please see info/svg-inkscape on CTAN:
%%   http://tug.ctan.org/tex-archive/info/svg-inkscape
%%
\begingroup%
  \makeatletter%
  \providecommand\color[2][]{%
    \errmessage{(Inkscape) Color is used for the text in Inkscape, but the package 'color.sty' is not loaded}%
    \renewcommand\color[2][]{}%
  }%
  \providecommand\transparent[1]{%
    \errmessage{(Inkscape) Transparency is used (non-zero) for the text in Inkscape, but the package 'transparent.sty' is not loaded}%
    \renewcommand\transparent[1]{}%
  }%
  \providecommand\rotatebox[2]{#2}%
  \newcommand*\fsize{\dimexpr\f@size pt\relax}%
  \newcommand*\lineheight[1]{\fontsize{\fsize}{#1\fsize}\selectfont}%
  \ifx\svgwidth\undefined%
    \setlength{\unitlength}{853.47930908bp}%
    \ifx\svgscale\undefined%
      \relax%
    \else%
      \setlength{\unitlength}{\unitlength * \real{\svgscale}}%
    \fi%
  \else%
    \setlength{\unitlength}{\svgwidth}%
  \fi%
  \global\let\svgwidth\undefined%
  \global\let\svgscale\undefined%
  \makeatother%
  \begin{picture}(1,0.32161459)%
    \lineheight{1}%
    \setlength\tabcolsep{0pt}%
    \put(0,0){\includegraphics[width=\unitlength,page=1]{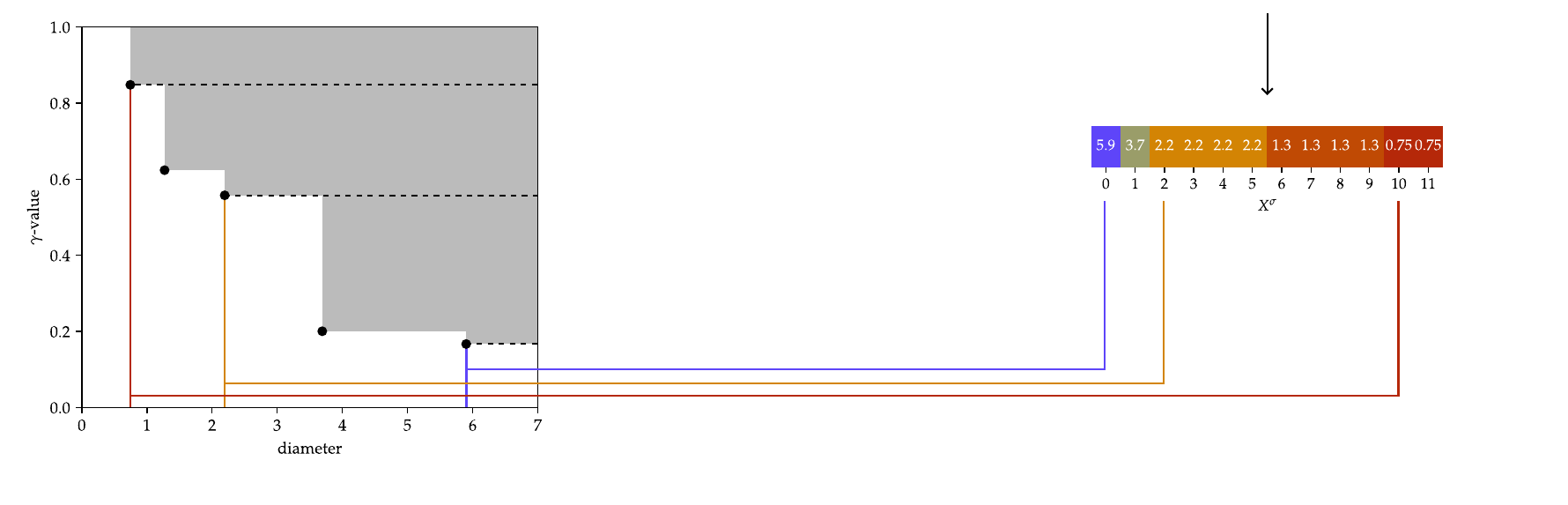}}%
    \put(0.3535961,0.09736868){\color[rgb]{0,0,0}\makebox(0,0)[lt]{\lineheight{1.25}\smash{\begin{tabular}[t]{l}$\gamma(x_0^\sigma)$\end{tabular}}}}%
    \put(0.3528047,0.19126245){\color[rgb]{0,0,0}\makebox(0,0)[lt]{\lineheight{1.25}\smash{\begin{tabular}[t]{l}$\gamma(x_2^\sigma)$\end{tabular}}}}%
    \put(0.35253527,0.26277419){\color[rgb]{0,0,0}\makebox(0,0)[lt]{\lineheight{1.25}\smash{\begin{tabular}[t]{l}$\gamma(x_{10}^\sigma)$\end{tabular}}}}%
    \put(0.82145529,0.28576147){\color[rgb]{0,0,0}\makebox(0,0)[lt]{\lineheight{1.25}\smash{\begin{tabular}[t]{l}$\max_{p\in P_\sigma}$\end{tabular}}}}%
  \end{picture}%
\endgroup%

  \caption[Tensor operations]{Illustration of the tensor operations that compute the set of bigrades of a triangle $\sigma\in\del(L)$.
  \textbf{Top left:} the sample $P_\sigma$ (black dots) has size 10 and the set of points $X^\sigma=\{x_0^\sigma,\ldots,x_{11}^\sigma\}$ has size 12 and is sorted and colored by $\gamma$ value.
  \textbf{Bottom left:} the support of $\sigma$ (in gray, with the minimal set of bigrades $\bigrades(\sigma)$ shown with black dots) is computed from the distance matrix $D_\sigma=\dist(P_\sigma,X^\sigma)$ to which we apply a cumulative minimum along each row, and then a maximum along each column.}
  \label{fig:sublevel_flood_computation}
\end{figure}

\section{Additional experiments}
\label{sec_additionalExperiments}

This section contains additional experiments that we did not include in the main paper.

\subsection{Runtimes}
\cref{fig:timings_additional} shows the computation runtime of several bifiltrations, without including the multiparameter module approximation (MMA) (unlike \cref{fig:timings} which includes it).

\begin{figure}
  \includegraphics[width=\linewidth]{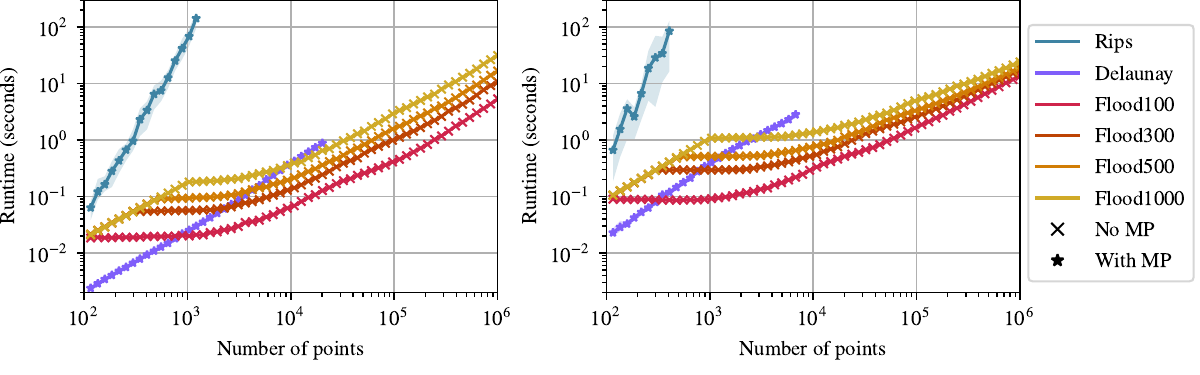}
  \caption{Running time of the bifiltration computation (plus its minimal presentation for Rips and Delaunay) for $X\subset\bbr^2$ (left) and $X\subset\bbr^3$ (right). Contrary to \cref{fig:timings}, this does not include multiparameter module approximation.}
  \label{fig:timings_additional}
\end{figure}

\subsection{Structure sizes}
\cref{fig:structure_sizes} shows the size of the bifiltrations (in terms of total number of bigrades) and the size of their multiparameter module approximation (in terms of number of summands in the module decomposition).
\changed{It also shows how the average criticality of the simplices in the sublevel Flood bifiltration increases with the number of points and decreases with the number of landmarks (sublevel Rips and Delaunay bifiltrations are not represented as their criticality is always 1).}

\begin{figure}
  \includegraphics[width=\linewidth]{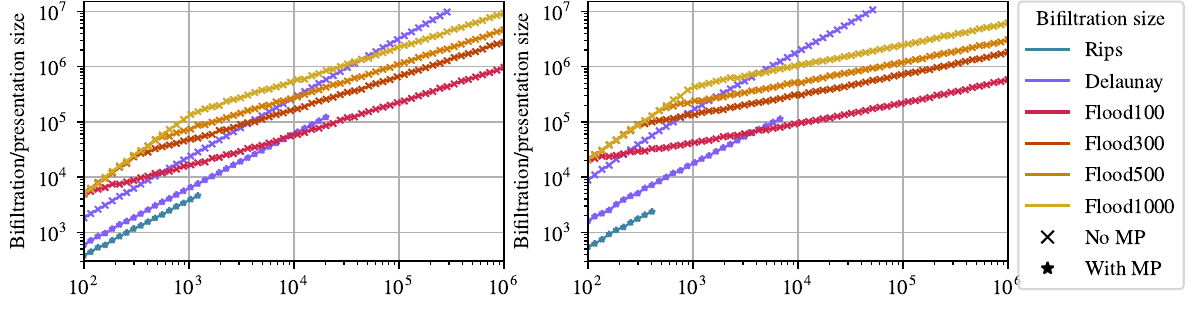}
  \includegraphics[width=\linewidth]{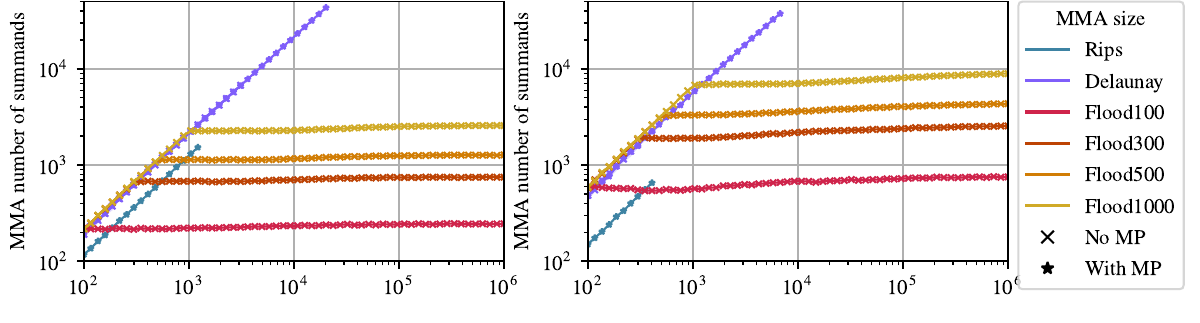}
  \includegraphics[width=\linewidth]{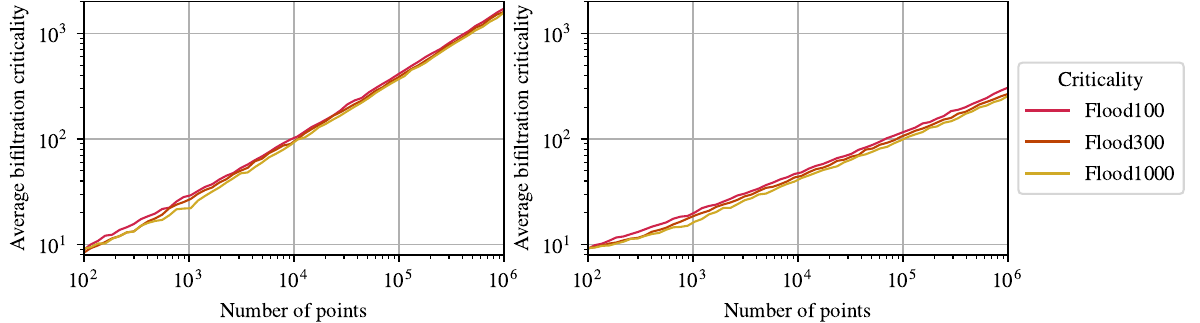}
  \caption{
  \textbf{Top row:} size of the bifiltrations, i.e., total number of stored bigrades. For 1-critical bifiltrations (all except sublevel Flood), this corresponds to the number of simplices. For the minimal presentations, the total number of generators is shown instead (all homology dimensions).
  \textbf{Middle row:} total number of summands in the multiparameter module approximation (all homology dimensions).
  \textbf{Bottom row:} average criticality $|\bigrades(\sigma)|$ in the sublevel Flood bifiltrations.}
  \label{fig:structure_sizes}
\end{figure}

\subsection{Empirical stability evaluation}
\cref{fig:stability} shows a study of the stability of the Flood bifiltration with respect to the input $X$, when $L$ is fixed (as in \cref{th:stability_X_gamma,th:stability_X}) and when $L$ is obtained separately from $X$ and its perturbation.

\begin{figure}
  \centering
  \includegraphics[width=.32\linewidth]{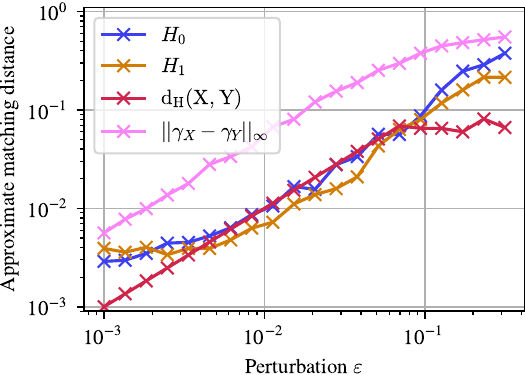}
  \includegraphics[width=.32\linewidth]{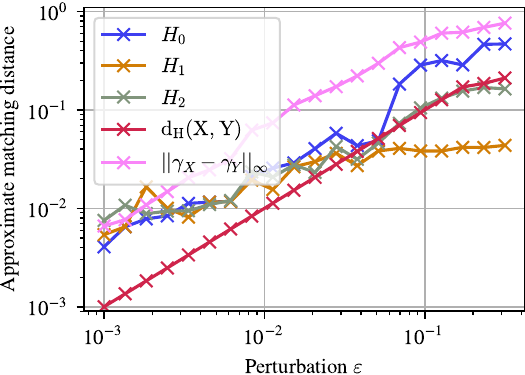}
  \includegraphics[width=.32\linewidth]{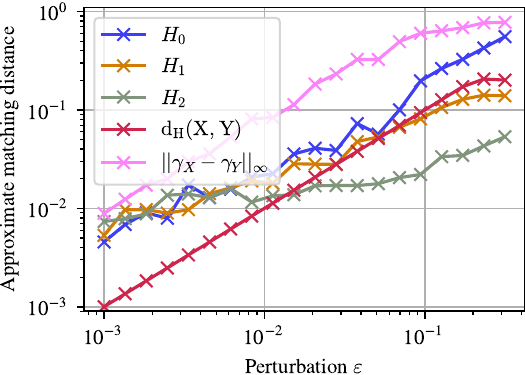}
  \includegraphics[width=.32\linewidth]{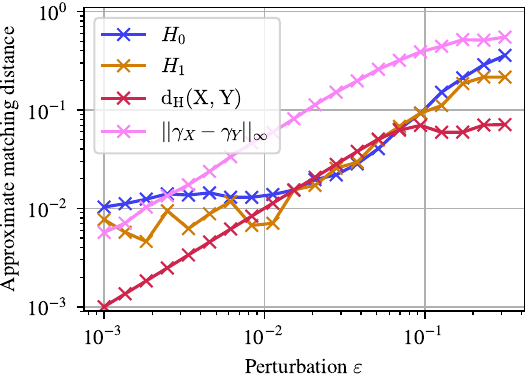}
  \includegraphics[width=.32\linewidth]{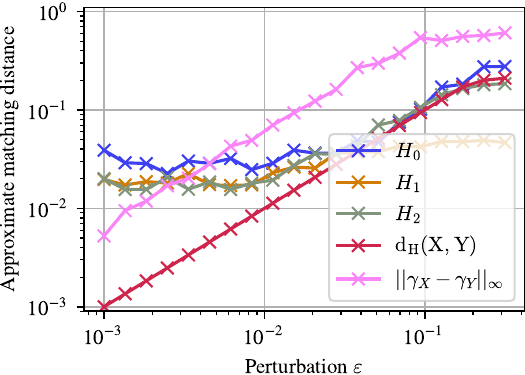}
  \includegraphics[width=.32\linewidth]{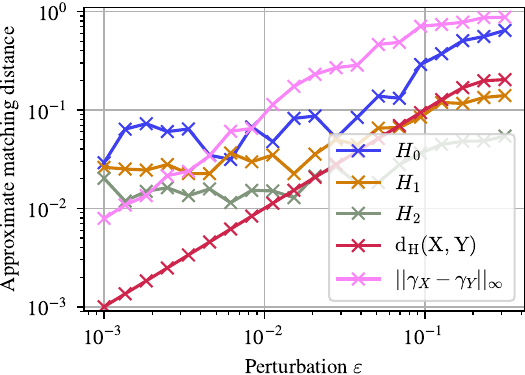}
  \caption{Stability of the sublevel Flood bifiltration with respect to $X$, as measured by the matching distance between $\floodmodule(X,L,\gamma_X)$ and $\floodmodule(X',L',\gamma_{X'})$ with $X'=X+E$ and with $\gamma_X$ and $\gamma_{X'}$ obtained as KDEs.
  In the top row, $L=L'$ is obtained as an FPS on $X$ (fixed landmarks, as in \cref{th:stability_X_gamma,th:stability_X}).
  In the bottom row, $L'$ is computed as an FPS from $X'$.
  $L$ has 1,000 points and $X$ has 10,000 points distributed over the noisy ring in $\bbr^2$ (left), the noisy sphere in $\bbr^3$ (center) and the noisy torus in $\bbr^3$ (right).}
  \label{fig:stability}
\end{figure}

\subsection{Additional classification results}
\cref{fig:swisscheese_2d_large} is similar to \cref{fig:synthetic_scores}, left, but with 20,000 points in $\bbr^2$ and 5,000 points in $\bbr^3$ per point set.
\changed{
  \cref{table:ucr_scores_additional} gives the accuracy scores on the UCR datasets that are not shown in the main paper.
}

\begin{figure}
  \centering
  \includegraphics[width=\linewidth]{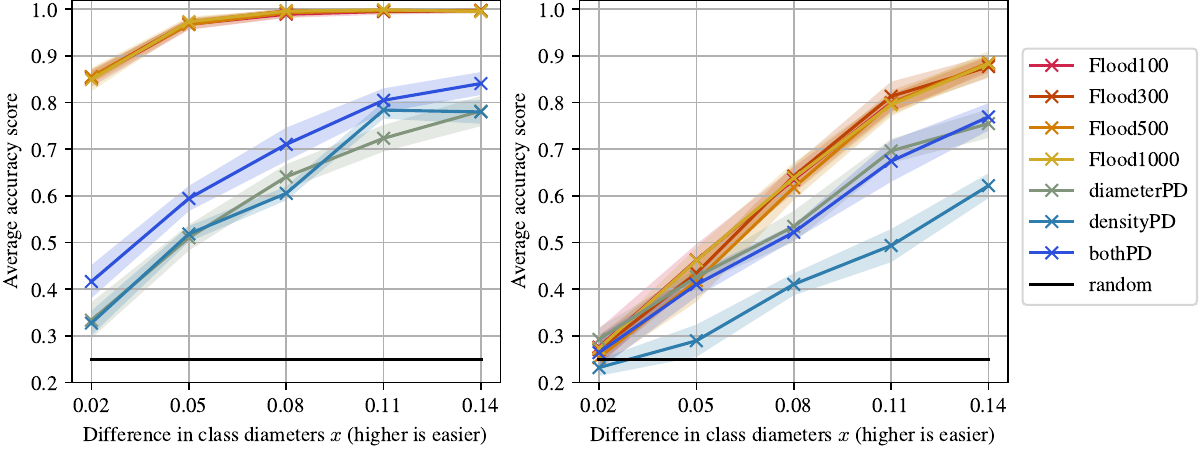}
  \caption{Accuracy scores (averaged over 10 splits) of single-parameter PH (diameterPD, densityPD and bothPD) and multiparameter PH (Flood) on the synthetic noisy swisscheese toy datasets parameterized by difficulty $x$. Here, the dataset contains 800 sets equally split into four classes characterized by the diameter of their low-density regions. The black lines show the score of a random classifier.
  \textbf{Left:} Each point set contains 20,000 points in $\bbr^2$.
  \textbf{Right:} Each point set contains 5,000 points in $\bbr^3$.
  }
  \label{fig:swisscheese_2d_large}
\end{figure}

\begin{table}
  \centering
  \caption{Accuracy scores on additional UCR time series datasets. $\deln$ and $\flood$ denote the Delaunay and Flood bifiltrations, respectively.}
  %\begin{tabular}{|l||rrr|r|rrrr|}
  \begin{tabular}{|l||rr|r|rrrr|}
  \hline
  % Dataset & ED & DTWl & DTWc & \cech\ & $\deln$/MPL & $\flood$/MPL & $\deln$/SCDR & $\flood$/SCDR \\
  Dataset & ED & DTWc & \cech\ & $\deln$/MPL & $\flood$/MPL & $\deln$/SCDR & $\flood$/SCDR \\
  \hline
\texttt{PPOAG} & 0.785 & 0.805 & 0.785 & \textbf{0.829} & 0.785 & 0.815 & 0.815\\
\texttt{PPOC} & 0.808 & 0.783 & \textbf{0.832} & 0.663 & 0.674 & 0.814 & 0.821\\
\texttt{PPTW} & 0.707 & 0.756 & 0.737 & \textbf{0.776} & 0.756 & 0.741 & 0.771\\
\texttt{ECG200} & \textbf{0.880} & 0.770 & 0.710 & 0.740 & 0.790 & 0.830 & 0.830\\
\texttt{MI} & 0.684 & \textbf{0.737} & 0.584 & 0.533 & 0.547 & 0.628 & 0.514\\
\texttt{Plane} & 0.962 & \textbf{1.000} & 0.876 & 0.629 & 0.790 & 0.933 & 0.610\\
\texttt{SL} & 0.789 & 0.792 & 0.738 & 0.656 & 0.552 & 0.842 & \textbf{0.864}\\
\texttt{GP} & 0.913 & 0.907 & 0.880 & 0.767 & 0.673 & 0.940 & \textbf{0.960}\\
\texttt{GPAS} & 0.899 & 0.918 & \textbf{0.949} & 0.522 & 0.494 & 0.940 & 0.937\\
\texttt{GPMVF} & 0.975 & \textbf{0.997} & 0.953 & 0.842 & 0.883 & 0.972 & 0.962\\
\texttt{GPOVY} & 0.952 & 0.838 & \textbf{1.000} & 0.994 & 0.625 & 0.590 & \textbf{1.000}\\
\texttt{SC} & 0.880 & \textbf{0.993} & 0.503 & 0.547 & 0.550 & 0.580 & 0.593\\
\texttt{EQ} & 0.712 & 0.719 & \textbf{0.763} & 0.647 & 0.705 & 0.755 & 0.741\\
\texttt{PAirP} & 0.058 & 0.106 & \textbf{0.264} & 0.043 & 0.024 & 0.077 & 0.048\\
\texttt{HT} & 0.664 & \textbf{0.924} & 0.815 & 0.672 & 0.580 & 0.395 & 0.807\\
  \hline
  \end{tabular}
  \label{table:ucr_scores_additional}
\end{table}

\subsection{Landmark sampling evaluation}
\label{sec:landmarks_sampling}
\changed{
We compare several sampling techniques to obtain a set of landmarks $L\subseteq X$ of prescribed size: furthest point sampling (FPS), random subsampling, and the density-aware approach (DA-FPS) proposed in~\cite{climaco2026densityaware} (see \cref{fig:samplings} for an illustration of these three approaches). \cref{table:sampling} evaluates those approaches by showing the associated matching distance to the reference sublevel Delaunay bifiltration on the noisy ring (in $\bbr^2$) and noisy sphere (in $\bbr^3$).
}

\begin{figure}
  \centering
  \begin{overpic}[width=.9\linewidth]{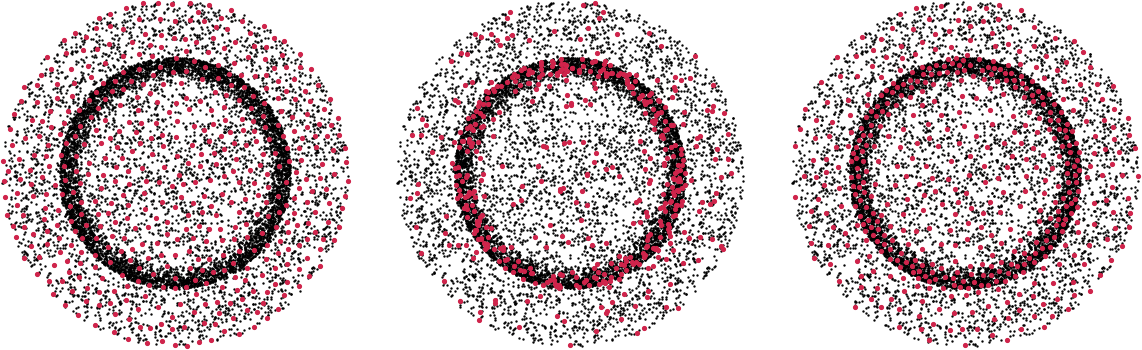}
    \put(13,-4){FPS}
    \put(46,-4){random}
    \put(80,-4){DA-FPS}
  \end{overpic}
  \vspace{.5cm}
  \caption{Three sampling approaches to get 500 landmarks (red dots) from an input $X\subset\bbr^2$ of size 10,000 (black dots).}
  \label{fig:samplings}
\end{figure}

\begin{table}
  \caption{Matching distance between the sublevel Delaunay and Flood bifiltrations for several landmark sampling techniques: classical and density-aware furthest point sampling, and random sampling. The input point sets have 10,000 points and 500 landmarks are selected. The matching distances are averaged over 100 runs.}
  \begin{tabular}{|r||rr|rrr|}
  \hline
         & \multicolumn{2}{c|}{Noisy ring ($\bbr^2$)}& \multicolumn{3}{c|}{Noisy sphere ($\bbr^3$)} \\
         & $\matchd$ ($H_0$)   & $\matchd$ ($H_1$)   & $\matchd$ ($H_0$)   & $\matchd$ ($H_1$)   & $\matchd$ ($H_2$)   \\
  \hline
  FPS    & $0.022\pm0.005$     & $0.020\pm0.005$     & $0.059\pm0.011$     & $0.052\pm0.005$     & $0.067\pm0.009$     \\
  random & $0.012\pm0.004$     & $0.012\pm0.004$     & $0.044\pm0.007$     & $0.032\pm0.011$     & $0.045\pm0.017$     \\
  DA-FPS & $0.012\pm0.003$     & $0.009\pm0.003$     & $0.049\pm0.008$     & $0.035\pm0.004$     & $0.038\pm0.007$     \\
  \hline
  \end{tabular}
  \label{table:sampling}
\end{table}

\subsection{Number of landmarks}
\label{sec:landmarks_number}

\changed{
To better illustrate how the number of landmarks has an influence in practice, we show how the classification accuracy score evolves when the number of landmarks increases for the toy synthetic dataset (in $\bbr^2$) in \cref{fig:failure_case}, left, and for the porous rock dataset (in $\bbr^3$) in \cref{fig:failure_case}, right.
Additionally, we use in \cref{fig:failure_case_2} a variant of the 2-dimensional toy synthetic dataset, with two classes whose point sets feature respectively 124 and 125 non-intersecting empty regions, each of radius 0.05, and placed randomly. The classification task involves determining which of these two classes the point set belongs to.
}

\changed{
The three datasets exhibit different behaviors.
For the toy synthetic dataset, few landmarks are required to obtain a good classification accuracy because low-density regions do not have to be captured individually to retrieve their diameter.
For the porous rock dataset, the classification accuracy increases more regularly with the number of landmarks, since smaller and smaller geometric details are captured in the point sets.
Finally, for the dataset depicted in \cref{fig:failure_case_2}, there is a sharp transition between low and high accuracy, because classifying correctly requires counting precisely the number of voids and therefore capturing them separately, which happens after a threshold number of landmarks is reached.
}

\begin{figure}
  \centering
  \includegraphics[width=\linewidth]{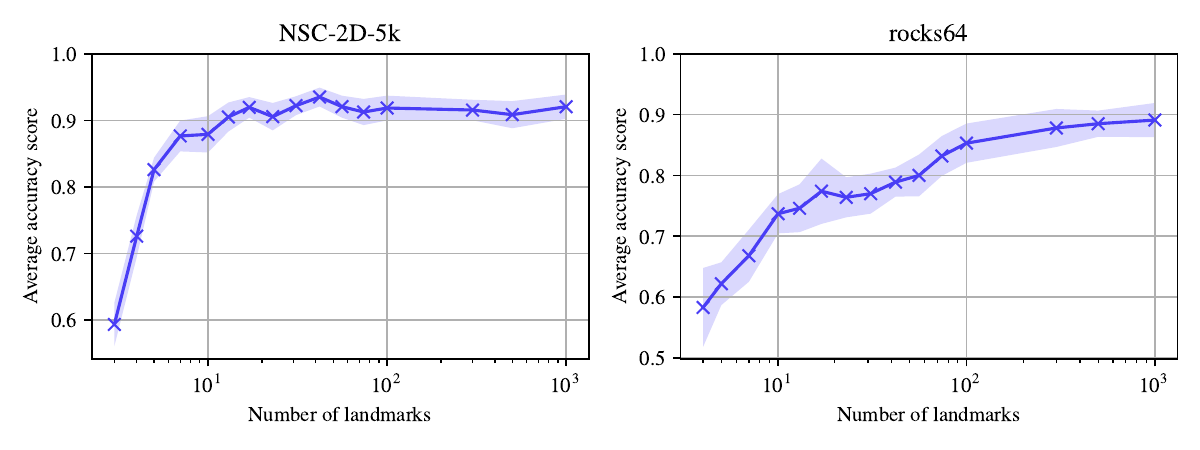}
  \caption{Accuracy score for two classification problems using the sublevel Flood bifiltration, as a function of the number of landmarks.}
  \label{fig:failure_case}
\end{figure}
\begin{figure}
  \centering
  \hspace{-7mm} 
  \begin{minipage}{.55\linewidth}
    \includegraphics[width=\linewidth]{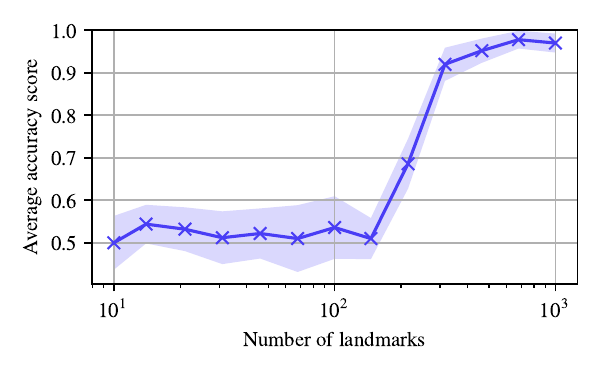}
  \end{minipage}
  \begin{minipage}{.43\linewidth}
    \begin{tabular}{cc}
      \includegraphics[width=.48\linewidth]{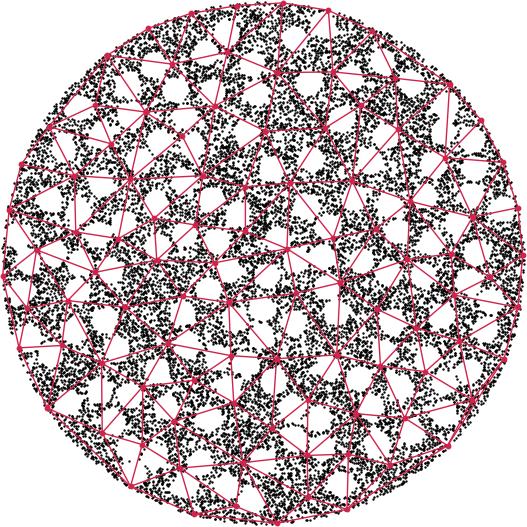} & \includegraphics[width=.48\linewidth]{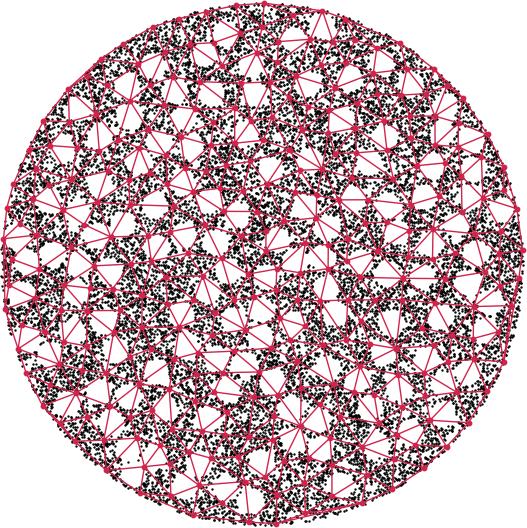} \\
      $|L|=100$ & $|L|=316$ \\
      $\text{avg. score}=0.54$ & $\text{avg. score}=0.95$
    \end{tabular}
  \end{minipage}
  \caption{Accuracy score for the classification problem involving counting the number of voids, using the sublevel Flood bifiltration, as a function of the number of landmarks. On the right, we show the Delaunay triangulation of the set of landmarks when the classification is unsuccessful ($|L|=100$) or successful ($|L|=316$).}
  \label{fig:failure_case_2}
\end{figure}

\changed{
  Finally, to provide a more visual insight into the influence of the number of landmarks, \cref{fig:mma_in_landmarks} shows the multiparameter module approximation of the sublevel Flood bifiltration as the number of landmarks increases and the sublevel Delaunay bifiltration is increasingly well approximated.
}

\begin{figure}
  \centering
  \includegraphics[width=\linewidth]{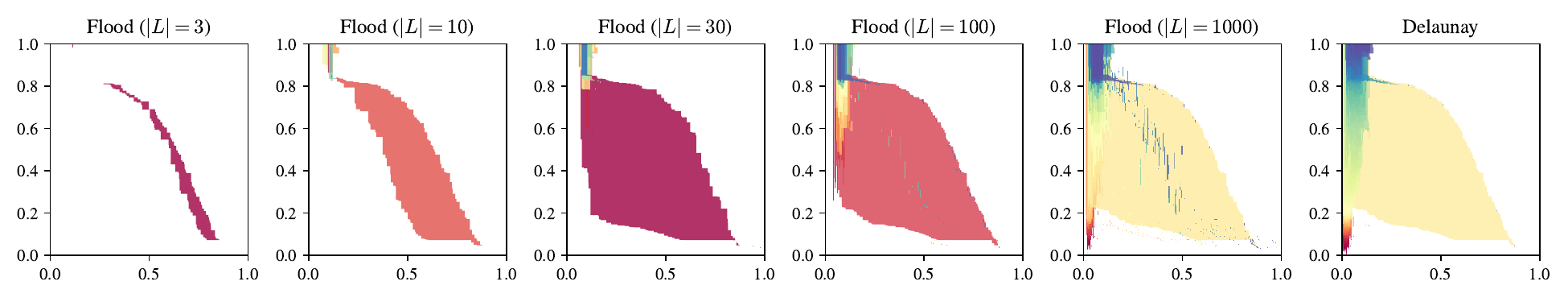}
  \caption{Multiparameter module approximation when the number of landmarks increases, for a point set $X$ of size 2,000 sampled as a noisy ring in $\bbr^2$.}
  \label{fig:mma_in_landmarks}
\end{figure}

\subsection{Interest of GPU use}
\label{sec:gpu}

\changed{
In order to clarify how the GPU batching helps to speed up the computation of the set of bigrades of the simplices, we show in \cref{fig:cpu_vs_gpu} the runtime comparison (bifiltration only, no MMA) between GPU (left number) and CPU (right number) on the 2-dimensional noisy ring and 3-dimensional noisy sphere datasets.
For this experiment, we used a Core i7-13850HX CPU (with 20 physical cores).
It illustrates that the use of the GPU is significantly profitable, especially starting from $\ddim=3$. This is because the higher number of samples (1540 samples for a tetrahedron against 210 for a triangle, with 20 samples per edge) makes the GPU batching more efficient.
Broadly speaking, our approach trades the costly PH computation on a large filtration for the costly but GPU-parallelizable construction of a much smaller one, on which PH is then much faster to compute.
}

\begin{figure}
  \centering
  \includegraphics[width=\linewidth]{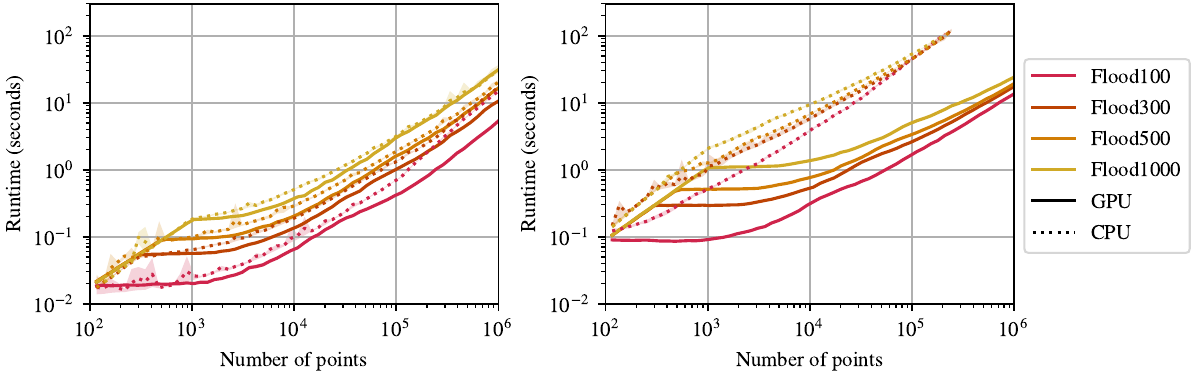}
  \caption{Comparison of runtimes for computing the sublevel Flood bifiltration (no MMA) on a GPU and a CPU (we used a Core i7-13850HX CPU for this experiment).}
  \label{fig:cpu_vs_gpu}
\end{figure}

\subsection{Masking strategy empirical evaluation}
\label{sec:masking}

\changed{
\cref{table:masking_evaluation} provides an empirical evaluation of the effect of the masking strategy described in \cref{sec:computational_approach} on the approximation (of the reference sublevel Delaunay bifiltration) and the runtime. A general observation is that---as expected---the exact version (i.e., taking the whole $X$ instead of $X^\sigma$ into account for all simplices $\sigma$) is significantly costlier in runtime and memory. In addition, empirically, the loss of precision of the approximate version compared to the exact one tends to be less significant when the size of $X$ increases. 
}

\begin{table}
  \centering
  \caption{Comparison between the exact (E) and masked (M) approaches (see \cref{sec:computational_approach}) in terms of runtime and approximation of the sublevel Delaunay bifiltration.}
  \begin{tabular}{|l|l|rr|rr|rr|rr|}
    \hline
    \multirow{2}{*}{$|X|$} & \multirow{2}{*}{$|L|$} & \multicolumn{2}{c|}{Runtime (s)} & \multicolumn{2}{c|}{$\matchd$ ($H_0$)} & \multicolumn{2}{c|}{$\matchd$ ($H_1$)} & \multicolumn{2}{c|}{$\matchd$ ($H_2$)} \\
    & & M & E & M & E & M & E & M & E \\ 
    \hline \multirow{4}{*}{$1000$}
    & $100$  & $0.14$ & $0.39$ & $0.17$ & $0.12$ & $0.04$ & $0.05$ & $0.05$ & $0.02$ \\
    & $300$  & $0.46$ & $1.38$ & $0.09$ & $0.09$ & $0.03$ & $0.03$ & $0.06$ & $0.01$ \\
    & $500$  & $0.76$ & $2.46$ & $0.07$ & $0.07$ & $0.03$ & $0.03$ & $0.06$ & $0.01$ \\
    & $1000$ & $1.70$ & $5.07$ & $0.01$ & $0.01$ & $0.03$ & $0.03$ & $0.02$ & $0.01$ \\
    \hline \multirow{4}{*}{$20000$}
    & $100$  & $0.55$ & $5.03$  & $0.05$ & $0.05$ & $0.08$ & $0.10$ & $0.10$ & $0.11$ \\
    & $300$  & $1.12$ & $18.85$ & $0.04$ & $0.05$ & $0.07$ & $0.06$ & $0.10$ & $0.08$ \\
    & $500$  & $1.65$ & $33.36$ & $0.03$ & $0.04$ & $0.06$ & $0.05$ & $0.08$ & $0.07$ \\
    & $1000$ & $2.99$ & $70.52$ & $0.03$ & $0.03$ & $0.05$ & $0.05$ & $0.05$ & $0.05$ \\
    \hline
  \end{tabular}
  \label{table:masking_evaluation}
\end{table}

\subsection{Codensity hyperparameter}

\changed{
The choice of the parameters of the estimators (e.g., bandwidth $h$ for KDE, $m$ for DTM) influences the accuracy in classification tasks, similarly to other methods. \cref{table:bandwidth} shows the accuracy obtained with the sublevel Flood (300 landmarks) and Delaunay bifiltrations on the 2-dimensional noisy swisscheese dataset with $x=0.05$ (see \cref{fig:synthetic_scores}, left) for Gaussian kernel density estimation with several bandwidths. We show the average over 10 splits. \cref{fig:synthetic_scores} used the same value 0.04 for all the methods.
}

\begin{table}
  \centering
  \caption{Average classification accuracy on the 2-dimensional swisscheese dataset with the sublevel Flood (300 landmarks) and Delaunay bifiltrations, using Gaussian KDE with several bandwidths $h$.}
 \begin{tabular}{|r|r|r|}
\hline
 $h$   & Flood (300 l.)       & Delaunay        \\
\hline
 $0.004$    & $0.793\pm0.027$ & $0.817\pm0.019$ \\
 $0.01$     & $0.891\pm0.016$ & $0.920\pm0.017$ \\
 $0.04$     & $0.916\pm0.015$ & $0.923\pm0.017$ \\
 $0.1$      & $0.774\pm0.029$ & $0.794\pm0.025$ \\
 $0.4$      & $0.429\pm0.035$ & $0.454\pm0.031$ \\
\hline
\end{tabular}
\label{table:bandwidth}
\end{table}

\section{Data description}
\label{sec_dataDescription}

This section contains a description of the synthetic and real datasets used throughout the paper.

\paragraph{Toy synthetic dataset}
Our first test case is a synthetic toy dataset, \texttt{noisy\_swiss\_cheese}, that aims to illustrate the typical case where multiparameter PH carries topological information that is not faithfully detected by either diameter-based or density-based single-parameter PH, or even by the combination of the two.
This dataset is made of points in the unit ball (in $\bbr^2$ and $\bbr^3$), uniformly sampled, except in a variable number $k\in\llbracket1,10\rrbracket$ of spherical regions that have a lower density.
For a given point set, these lower-density regions have the same diameter $\delta$ among $\{\delta_0,\delta_0+x/3,\delta_0+2x/3,\delta_0+x\}$, with $\delta_0=0.1$ and $x>0$.
The classification task involves determining which of the four diameter classes the point set belongs to.
The parameter $x$ quantifies the difference in class diameters and the problem's difficulty: as $x$ increases, the classes become easier to separate.
We only use $H_1$ (for $\bbr^2$) or $H_2$ (for $\bbr^3$) for all the methods.
Results are shown in \cref{fig:synthetic_scores}, left.

\paragraph{Porous rocks dataset}
The second dataset models porous material and is generated with the \texttt{PoreSpy} library~\cite{gostick2019porespy}.
It generates binary $64^3$ grids that we convert into a point set, with some Gaussian noise (to avoid the non-generic position due to the lattice structure).
The porosity is chosen randomly, uniformly in $[0.05,0.95]$.
The classification task consists of determining which porous material generator was used, between:
\texttt{blobs} that generates random noise combined with Gaussian blur, with a blobiness parameter chosen uniformly in $[0.5,2]$;
and \texttt{fractals} which is simplex noise with a frequency chosen uniformly in $[0.05,0.2]$.
To measure the ability of multiparameter PH to handle noise and outliers, we parameterize the difficulty of the problem with an ambient noise parameter which represents the probability $p$ of each element of the binary grid to be set to 1 (after the binary generation and before the point set conversion). We only use $H_2$.
Results are shown in \cref{fig:synthetic_scores}, right.

\paragraph{Time series datasets}
We perform experiments involving time series classification. TDA has been applied to such problems~\cite{perea2015sliding} by applying PH to the time-delay Takens embeddings~\cite{takens2006detecting} of the time series. Indeed, signal properties (e.g., periodicity) may translate into specific topological features in the embedding (e.g., loops).
Multiparameter PH was also recently used in this context (see, e.g., \cite{carriere2020multiparameter,loiseaux2023framework}), with the idea that noise in the time series transmits to the embedding.
We take several time series datasets from the UCR archive~\cite{dau2019ucr}
%\footnote{available at \url{https://www.timeseriesclassification.com}}
of various sizes, lengths, numbers of classes and coming from different scientific fields (see \cref{table:ucr_datasets} for the description of the datasets used) and use the suggested train and test datasets.
We encode them into $\bbr^3$, with time-delay Takens embedding, and apply the same pipelines as above, considering $H_0$ and $H_1$, with a DTM with parameter $m=0.1$ as codensity estimation.
Results are shown in \cref{table:ucr_scores,table:ucr_scores_additional},
in which we also report the baseline scores provided by~\cite{dau2019ucr} (except for \texttt{KLC}, \texttt{CD} and \texttt{BH}, which we ran ourselves as they do not belong to the original UCR dataset):
%nearest neighbor with Euclidean distance (ED), and dynamic time warping with learned (DTWl) and constant (DTWc) window width.
nearest neighbor with Euclidean distance (ED), and dynamic time warping with constant window width (DTWc).

\begin{figure}
  \centering
  \begin{overpic}[angle=90,width=\linewidth]{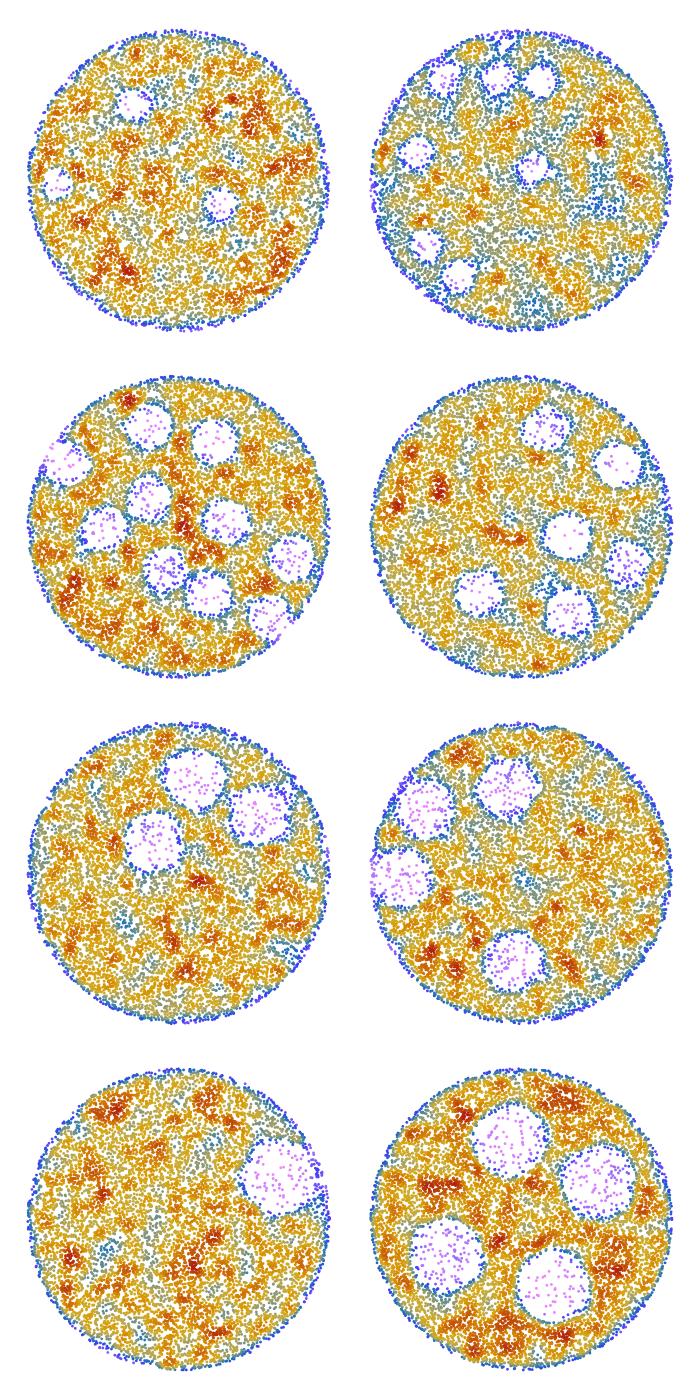}
    \put(9,50){$\delta=0.1$}
    \put(32.5,50){$\delta=0.147$}
    \put(57,50){$\delta=0.193$}
    \put(83,50){$\delta=0.24$}
  \end{overpic}
  \caption{Two examples of each class in the \texttt{noisy\_swiss\_cheese} dataset in $\bbr^2$, with 10,000 points in each point set and difficulty parameter $x=0.14$.}
\end{figure}

\begin{table}
  \centering
    \caption{The list of the UCR datasets used in this paper, with their acronym and their full name, the size of train and test sets, the length of the time series and the number of classes.}
  \begin{tabular}{|l|l|rrrrl|}
  \hline
  Dataset   & UCR dataset full name          &   \rotatebox{90}{Train} &   \rotatebox{90}{Test} &   \rotatebox{90}{Length} &   \rotatebox{90}{Classes} & Category     \\
  \hline
  \texttt{DPOAG}     & DistalPhalanxOutlineAgeGroup   &     400 &    139 &       80 &         3 & image        \\
  \texttt{DPOC}      & DistalPhalanxOutlineCorrect    &     600 &    276 &       80 &         2 & image        \\
  \texttt{DPTW}      & DistalPhalanxTW                &     400 &    139 &       80 &         6 & image        \\
  \texttt{PPOAG}     & ProximalPhalanxOutlineAgeGroup &     400 &    205 &       80 &         3 & image        \\
  \texttt{PPOC}      & ProximalPhalanxOutlineCorrect  &     600 &    291 &       80 &         2 & image        \\
  \texttt{PPTW}      & ProximalPhalanxTW              &     400 &    205 &       80 &         6 & image        \\
  \texttt{ECG200}    & ECG200                         &     100 &    100 &       96 &         2 & ecg          \\
  \texttt{IPD}       & ItalyPowerDemand               &      67 &   1029 &       24 &         2 & sensor       \\
  \texttt{MI}        & MedicalImages                  &     381 &    760 &       99 &        10 & image        \\
  \texttt{Plane}     & Plane                          &     105 &    105 &      144 &         7 & sensor       \\
  \texttt{SL}        & SwedishLeaf                    &     500 &    625 &      128 &        15 & image        \\
  \texttt{GP}        & GunPoint                       &      50 &    150 &      150 &         2 & har          \\
  \texttt{GPAS}      & GunPointAgeSpan                &     135 &    316 &      150 &         2 & har          \\
  \texttt{GPMVF}     & GunPointMaleVersusFemale       &     135 &    316 &      150 &         2 & har          \\
  \texttt{GPOVY}     & GunPointOldVersusYoung         &     136 &    315 &      150 &         2 & har          \\
  \texttt{PC}        & PowerCons                      &     180 &    180 &      144 &         2 & device       \\
  \texttt{SC}        & SyntheticControl               &     300 &    300 &       60 &         6 & simulated    \\
  \texttt{EQ}        & Earthquakes                    &     322 &    139 &      512 &         2 & sensor       \\
  \texttt{LN2}       & Lightning2                     &      60 &     61 &      637 &         2 & sensor       \\
  \texttt{Haptics}   & Haptics                        &     155 &    308 &     1092 &         5 & motion       \\
  \texttt{ACSF1}     & ACSF1                          &     100 &    100 &     1460 &        10 & device       \\
  \texttt{PigCVP}    & PigCVP                         &     104 &    208 &     2000 &        52 & hemodynamics \\
  \texttt{PArtP}     & PigArtPressure                 &     104 &    208 &     2000 &        52 & hemodynamics \\
  \texttt{PAirP}     & PigAirwayPressure              &     104 &    208 &     2000 &        52 & hemodynamics \\
  \texttt{HT}        & HouseTwenty                    &      40 &    119 &     2000 &         2 & device       \\
  \texttt{Rock}      & Rock                           &      20 &     50 &     2844 &         4 & spectro      \\
  \texttt{KLC}       & KeplerLightCurves              &     920 &    399 &     4767 &         7 & sensor       \\
  \texttt{CD}        & CatsDogs                       &     164 &    164 &    14773 &         2 & audio        \\
  \texttt{BH}        & BinaryHeartbeat                &     204 &    205 &    18530 &         2 & audio        \\
  \hline
  \end{tabular}
  \label{table:ucr_datasets}
\end{table}

\end{document}